\documentclass[12pt,reqno]{amsart}
\usepackage{graphicx}
\usepackage{amsmath}
\usepackage{amssymb}
\usepackage{amscd}
\usepackage{hhline}
\usepackage{mathrsfs}
 \usepackage{color}
 \usepackage{ifthen}
\usepackage{tikz}
\usepackage{framed}
\def\xstringversion     {1.6}
\def\xstringdate        {2012/10/24}

\edef\CurrentAtCatcode  {\the\catcode`\@}
\catcode`\@=11
\newwrite\@xs@message
\newcount\integerpart\newcount\decimalpart
\newif\if@xs@empty

\expandafter\ifx\csname @latexerr\endcsname\relax
	\immediate\write\m@ne{Package: xstring \xstringdate\space\space v\xstringversion\space\space String manipulations (C Tellechea)}%
	\long\def\@firstoftwo#1#2{#1}
	\long\def\@secondoftwo#1#2{#2}
	\long\def\@gobble#1{}
	\long\def\@ifnextchar#1#2#3{%
		\let\reserved@d=#1%
		\def\reserved@a{#2}%
		\def\reserved@b{#3}%
		\futurelet\@let@arg\@ifnch}
	\def\@ifnch{%
		\ifx\@let@arg\@sptoken
			\let\reserved@c\@xifnch
		\else
			\ifx\@let@arg\reserved@d
				\let\reserved@c\reserved@a
			\else
				\let\reserved@c\reserved@b
			\fi
		\fi
		\reserved@c}
	\def\:{\let\@sptoken= } \:
	\def\:{\@xifnch} \expandafter\def\: {\futurelet\@let@arg\@ifnch}
	\def\@ifstar#1{\@ifnextchar *{\@firstoftwo{#1}}}
	\long\def\@testopt#1#2{\@ifnextchar[{#1}{#1[{#2}]}}
	\def\@empty{}
\fi

\def\@xs@testempty#1{%
	\expandafter\ifx\expandafter\@empty\detokenize{#1}\@empty\@xs@emptytrue\else\@xs@emptyfalse\fi}

\def\@xs@MakeVerb{
	\begingroup
		\def\do##1{\catcode`##112\relax}%
		\dospecials
		\obeyspaces
		\@xs@ReadVerb
}

\def\setverbdelim#1{
	\expandafter\@xs@testempty\expandafter{\@gobble#1}%
	\if@xs@empty
	\else
		\begingroup
			\newlinechar`\^^J%
			\immediate\write\@xs@message
			{Package xstring Warning: verb delimiter is not a single token on input line \the\inputlineno^^J}%
		\endgroup
	\fi
	\def\@xs@ReadVerb##1#1##2#1{
		\endgroup
		\@xs@afterreadverb{##2}}
}

\def\verbtocs#1{%
	\def\@xs@afterreadverb##1{\def#1{##1}}%
	\@xs@MakeVerb
}

\begingroup
	\catcode\z@12 
	\gdef\tokenize#1#2{%
		\begingroup
			\@xs@def\@xs@reserved@A{#2}
			\def\@xs@AssignResult^^00##1^^00\@xs@nil{\gdef#1{##1}}
			\everyeof{\@xs@nil}
			\endlinechar\m@ne
			\catcode\z@12\relax
			\expandafter\@xs@AssignResult\scantokens\expandafter{\expandafter^^00\@xs@reserved@A^^00}
		\endgroup
	}%
\endgroup

\begingroup
	\catcode\z@3 \def\@xs@twochars{^^00}%
	\catcode\z@7 \xdef\@xs@twochars{\@xs@twochars^^00}%
\endgroup

\edef\@xs@reserved@A{\def\noexpand\@xs@AssignResult##1\@xs@twochars}
\@xs@reserved@A#2{\endgroup\expandafter\def\expandafter#2\expandafter{\@gobble#1}}
\def\tokenize#1#2{%
	\begingroup
		\@xs@def\@xs@reserved@A{#2}
		\everyeof\expandafter{\@xs@twochars#1}
		\endlinechar\m@ne
		\expandafter\@xs@AssignResult\scantokens\expandafter{\expandafter\relax\@xs@reserved@A}
}%

\def\@xs@ReturnResult#1#2{%
	\def\@xs@argument@A{#1}%
	\@xs@testempty{#2}%
	\if@xs@empty
		\@xs@argument@A
	\else
		\let#2\@xs@argument@A
	\fi
}

\def\normalexpandarg{%
	\let\@xs@def\def
	\def\@xs@expand##1{\unexpanded\expandafter{##1}}}

\def\expandarg{%
	\let\@xs@def\def
	\def\@xs@expand##1{\unexpanded\expandafter\expandafter\expandafter{##1}}%
}

\def\fullexpandarg{%
	\let\@xs@def\edef
	\def\@xs@expand##1{##1}
}

\def\saveexpandmode{\let\@xs@saved@def\@xs@defarg\let\@xs@saved@expand\@xs@expand}
\def\restoreexpandmode{\let\@xs@defarg\@xs@saved@def\let\@xs@expand\@xs@saved@expand}

\def\@xs@ifbeginwithbrace#1{%
	\csname @%
		\expandafter\@gobble\string{%
		\expandafter\@gobble\expandafter{\expandafter{\string#1}%
		\expandafter\expandafter\expandafter\expandafter\expandafter\expandafter\expandafter\expandafter\expandafter\expandafter\expandafter\expandafter\expandafter\expandafter\expandafter\@firstoftwo
		\expandafter\expandafter\expandafter\expandafter\expandafter\expandafter\expandafter\@gobble
		\expandafter\expandafter\expandafter\@gobble
		\expandafter\expandafter\expandafter{\expandafter\string\expandafter}\string}%
		\expandafter\@gobble\string}%
		\@secondoftwo{first}{second}oftwo%
	\endcsname
}

\def\@xs@returnfirstsyntaxunit#1#2{%
	\def\@xs@groupfound{\expandafter\def\expandafter#2\expandafter{\expandafter{#2}}\@xs@gobbleall}
	\def\@xs@assignfirsttok##1##2\@xs@nil{\let\@xs@toks0\def#2{##1}}%
	\def\@xs@testfirsttok{%
		\let\@xs@next\@xs@assignfirsttok
		\ifx\@xs@toks\bgroup
			\expandafter\@xs@ifbeginwithbrace\expandafter{\@xs@argument}{\def\@xs@next{\afterassignment\@xs@groupfound\def#2}}{}%
		\fi
		\@xs@next}%
	\def\@xs@argument{#1}%
	\edef\@xs@next{\expandafter\@xs@beforespace\detokenize{#1} \@xs@nil}
	\ifx\@xs@next\@empty
		\def\@xs@next{\expandafter\ifx\expandafter\@empty\detokenize\expandafter{\@xs@argument}\@empty\let#2\@empty\else\def#2{ }\let\@xs@toks0\fi}%
	\else
		\def\@xs@next{\expandafter\futurelet\expandafter\@xs@toks\expandafter\@xs@testfirsttok\@xs@argument\@xs@nil}%
	\fi
	\@xs@next
}

\def\@xs@testsecondtoken#1\@xs@nil{\@xs@ifbeginwithbrace{#1}}
\def\@xs@gobblespacebeforebrace#1#{}
\def\@xs@removefirstsyntaxunit#1#2{%
	\def\@xs@argument{#1}%
	\expandafter\expandafter\expandafter\ifx\expandafter\expandafter\expandafter\@empty\expandafter\@xs@beforespace\detokenize\expandafter{\@xs@argument} \@xs@nil\@empty
		\expandafter\@xs@testempty\expandafter{\@xs@argument}%
		\if@xs@empty
			\let#2\@empty
		\else
			\afterassignment\@xs@testsecondtoken
			\expandafter\let\expandafter\@xs@secontoken\expandafter=\expandafter\@sptoken\@xs@argument\@xs@@nil\@xs@nil
				{\expandafter\expandafter\expandafter\def\expandafter\expandafter\expandafter#2%
				 \expandafter\expandafter\expandafter{\expandafter\@xs@gobblespacebeforebrace\@xs@argument}}%
				{\expandafter\expandafter\expandafter\def\expandafter\expandafter\expandafter#2%
				 \expandafter\expandafter\expandafter{\expandafter\@xs@behindspace\@xs@argument\@xs@nil}}%
		\fi
	\else
		\expandafter\expandafter\expandafter\def\expandafter\expandafter\expandafter#2%
		\expandafter\expandafter\expandafter{\expandafter\@gobble\@xs@argument}%
	\fi
}

\def\@xs@beforespace#1 #2\@xs@nil{#1}
\def\@xs@behindspace#1 #2\@xs@nil{#2}
\def\@xs@returnfirstsyntaxunit@ii#1#2\@xs@nil{#1}
\def\@xs@gobbleall#1\@xs@nil{}

\def\@xs@expand@and@detokenize#1#2{%
	\def#1{#2}%
	\expandafter\edef\expandafter#1\expandafter{\@xs@expand#1}
	\edef#1{\detokenize\expandafter{#1}}
}

\def\@xs@expand@and@assign#1#2{\@xs@def#1{#2}}

\def\@xs@edefaddtomacro#1#2{\edef#1{\unexpanded\expandafter{#1}#2}}
\def\@xs@addtomacro#1#2{\expandafter\def\expandafter#1\expandafter{#1#2}}

\def\@xs@argstring{0########1########2########3########4########5########6########7########8########9}
\def\@xs@DefArg#1{\def\@xs@defarg0##1#1##2\@xs@nil{\def\@xs@myarg{##1#1}}\expandafter\@xs@defarg\@xs@argstring\@xs@nil}
\def\@xs@DefArg@#1{\expandafter\@xs@defarg@\expandafter{\number\numexpr#1+1}}
\def\@xs@defarg@#1{\def\@xs@defarg0##11##2#1##3\@xs@nil{\def\@xs@myarg{[##11]##2#1}}\expandafter\@xs@defarg\@xs@argstring\@xs@nil}
\def\@xs@OneArg#1{\expandafter\@xs@onearg\expandafter{\number\numexpr#1-1}{#1}}
\def\@xs@onearg#1#2{\def\@xs@defarg##1#1##2#2##3\@xs@nil{\def\@xs@myarg{##2#2}}\expandafter\@xs@defarg\@xs@argstring\@xs@nil}

\def\@xs@BuildLines#1#2#3#4{%
	\let\@xs@newlines\@empty
	\let\@xs@newargs\@empty
	\def\@xs@buildlines##1{%
		\expandafter\@xs@OneArg\expandafter{\number\numexpr##1+#1-1}%
		\edef\@xs@reserved@B{\noexpand\@xs@expand\csname @xs@arg@\romannumeral\numexpr##1\endcsname}%
		\ifnum##1=\@ne
			\@xs@testempty{#3}%
			\if@xs@empty
				\expandafter\@xs@addtomacro\expandafter\@xs@newargs\expandafter{\expandafter{\@xs@reserved@B}}%
				\edef\@xs@reserved@B{\ifnum##1>#4 @xs@def\else @xs@assign\fi}%
			\else
				\expandafter\@xs@addtomacro\expandafter\@xs@newargs\expandafter{\expandafter[\@xs@reserved@B]}%
				\def\@xs@reserved@B{@xs@def}
			\fi
		\else
			\expandafter\@xs@addtomacro\expandafter\@xs@newargs\expandafter{\expandafter{\@xs@reserved@B}}%
			\edef\@xs@reserved@B{\ifnum##1>#4 @xs@def\else @xs@assign\fi}%
		\fi
		\edef\@xs@newlines{\unexpanded\expandafter{\@xs@newlines}\expandafter\noexpand\csname\@xs@reserved@B\endcsname\expandafter\noexpand\csname @xs@arg@\romannumeral\numexpr##1\endcsname{\@xs@myarg}}%
		\ifnum##1<#2\relax
			\def\@xs@next{\expandafter\@xs@buildlines\expandafter{\number\numexpr##1+1}}%
			\expandafter\@xs@next
		\fi}%
	\@xs@buildlines\@ne
}

\def\@xs@newmacro{%
	\@ifstar
		{\let\@xs@reserved@D\@empty\@xs@newmacro@}
		{\let\@xs@reserved@D\relax\@xs@newmacro@0}%
}
\def\@xs@newmacro@#1#2#3#4#5{%
	\edef\@xs@reserved@A{@xs@\expandafter\@gobble\string#2}%
	\edef\@xs@reserved@C{\expandafter\noexpand\csname\@xs@reserved@A @\ifx\@empty#3\@empty @\fi\endcsname}%
	\edef\@xs@reserved@B{%
		\ifx\@empty\@xs@reserved@D
			\def\noexpand#2{\noexpand\@ifstar
				{\let\noexpand\@xs@assign\noexpand\@xs@expand@and@detokenize\expandafter\noexpand\@xs@reserved@C}%
				{\let\noexpand\@xs@assign\noexpand\@xs@expand@and@assign\expandafter\noexpand\@xs@reserved@C}%
			}%
		\else
			\def\noexpand#2{\let\noexpand\@xs@assign\noexpand\@xs@expand@and@assign\expandafter\noexpand\@xs@reserved@C}%
		\fi
		\ifx\@empty#3\@empty
		\else
			\def\expandafter\noexpand\@xs@reserved@C{%
				\noexpand\@testopt{\expandafter\noexpand\csname\@xs@reserved@A @@\endcsname}{\ifx\@xs@def\edef#3\else\unexpanded{#3}\fi}}%
		\fi
	}%
	\@xs@reserved@B
	\ifx\@empty#3\@empty
		\@xs@BuildLines1{#4}{#3}{#1}%
		\@xs@DefArg{#4}%
	\else
		\expandafter\@xs@BuildLines\expandafter1\expandafter{\number\numexpr#4+1}{#3}{#1}%
		\@xs@DefArg@{#4}%
	\fi
	\edef\@xs@reserved@B{\def\expandafter\noexpand\csname\@xs@reserved@A @@\endcsname\@xs@myarg}%
	\edef\@xs@reserved@C{\unexpanded\expandafter{\@xs@newlines}\edef\noexpand\@xs@call}%
	\edef\@xs@reserved@D{%
		\noexpand\noexpand\expandafter\noexpand\csname\@xs@reserved@A\endcsname\unexpanded\expandafter{\@xs@newargs}%
	}%
	\ifnum#5=\@ne\edef\@xs@reserved@D{\noexpand\noexpand\noexpand\@testopt{\unexpanded\expandafter{\@xs@reserved@D}}{}}\fi
	\@xs@edefaddtomacro\@xs@reserved@C{{\unexpanded\expandafter{\@xs@reserved@D}}\noexpand\@xs@call}%
	\@xs@edefaddtomacro\@xs@reserved@B{{\unexpanded\expandafter{\@xs@reserved@C}}}%
	\@xs@reserved@B
	\edef\@xs@reserved@B{%
		\def\expandafter\noexpand\csname\@xs@reserved@A\endcsname
			\@xs@myarg\ifnum#5=\@ne[\unexpanded{##}\number\numexpr\ifx\@empty#3\@empty#4+1\else#4+2\fi]\fi
	}%
	\@xs@reserved@B
}

\def\@xs@read@reserved@C{%
	\expandafter\@xs@testempty\expandafter{\@xs@reserved@C}%
	\if@xs@empty
		\ifnum\@xs@nestlevel=\z@
			\let\@xs@next\relax
		\else
			\let\@xs@next\@xs@atendofgroup
		\fi
	\else
		\expandafter\@xs@returnfirstsyntaxunit\expandafter{\@xs@reserved@C}\@xs@reserved@A
		\expandafter\@xs@removefirstsyntaxunit\expandafter{\@xs@reserved@C}\@xs@reserved@C
		\let\@xs@next\@xs@read@reserved@C
		\@xs@exploregroups
		\ifx\bgroup\@xs@toks
			\advance\integerpart\@ne
			\begingroup
				\expandafter\def\expandafter\@xs@reserved@C\@xs@reserved@A
				\@xs@manage@groupID
				\let\@xs@nestlevel\@ne
				\integerpart\z@
				\@xs@atbegingroup
		\else
			\global\advance\decimalpart\@ne
			\@xs@atnextsyntaxunit
		\fi
	\fi
	\@xs@next
}

\def\@xs@read@reserved@D{%
	\expandafter\@xs@testempty\expandafter{\@xs@reserved@D}%
	\if@xs@empty
		\ifnum\@xs@nestlevel=\z@
			\let\@xs@next\relax
		\else
			\let\@xs@next\@xs@atendofgroup
		\fi
	\else
		\expandafter\expandafter\expandafter\@xs@IfBeginWith@i\expandafter\expandafter\expandafter{\expandafter\@xs@reserved@D\expandafter}\expandafter{\@xs@reserved@E}%
			{\global\advance\decimalpart\@ne
			\let\@xs@reserved@D\@xs@reserved@A
			\@xs@atoccurfound
			}%
			{\expandafter\@xs@returnfirstsyntaxunit\expandafter{\@xs@reserved@D}\@xs@reserved@A
			\expandafter\@xs@removefirstsyntaxunit\expandafter{\@xs@reserved@D}\@xs@reserved@D
			\let\@xs@next\@xs@read@reserved@D
			\@xs@exploregroups
			\ifx\bgroup\@xs@toks
				\advance\integerpart\@ne
				\begingroup
					\expandafter\def\expandafter\@xs@reserved@D\@xs@reserved@A
					\@xs@manage@groupID
					\let\@xs@reserved@C\@empty
					\let\@xs@nestlevel\@ne
					\integerpart\z@
			\else
				\expandafter\@xs@addtomacro\expandafter\@xs@reserved@C\expandafter{\@xs@reserved@A}%
			\fi
			}%
	\fi
	\@xs@next
}

\@xs@newmacro\StrRemoveBraces{}{1}{1}{%
	\def\@xs@reserved@C{#1}%
	\let\@xs@reserved@B\@empty
	\let\@xs@nestlevel\z@
	\@xs@StrRemoveBraces@i
	\expandafter\@xs@ReturnResult\expandafter{\@xs@reserved@B}{#2}%
}

\def\@xs@StrRemoveBraces@i{%
	\expandafter\@xs@testempty\expandafter{\@xs@reserved@C}%
	\if@xs@empty
		\ifnum\@xs@nestlevel=\z@
			\let\@xs@next\relax
		\else
			\expandafter\endgroup
			\expandafter\@xs@addtomacro\expandafter\@xs@reserved@B\expandafter{\@xs@reserved@B}%
			\let\@xs@next\@xs@StrRemoveBraces@i
		\fi
	\else
		\expandafter\@xs@returnfirstsyntaxunit\expandafter{\@xs@reserved@C}\@xs@reserved@A
		\expandafter\@xs@removefirstsyntaxunit\expandafter{\@xs@reserved@C}\@xs@reserved@C
		\let\@xs@next\@xs@StrRemoveBraces@i
		\ifx\bgroup\@xs@toks
			\ifx\@xs@exploregroups\relax
				\begingroup
					\expandafter\def\expandafter\@xs@reserved@C\@xs@reserved@A
					\let\@xs@nestlevel\@ne
					\integerpart\z@
					\let\@xs@reserved@B\@empty
			\else
				\expandafter\@xs@addtomacro\expandafter\@xs@reserved@B\@xs@reserved@A
			\fi
		\else
			\expandafter\@xs@addtomacro\expandafter\@xs@reserved@B\expandafter{\@xs@reserved@A}%
		\fi
	\fi
	\@xs@next
}

\def\@xs@cutafteroccur#1#2#3{%
	\ifnum#3<\@ne\expandafter\@firstoftwo\else\expandafter\@secondoftwo\fi
		{\let\@xs@reserved@C\@empty\let\@xs@reserved@E\@empty\let\groupID\@empty}
		{\@xs@cutafteroccur@i{#1}{#2}{#3}}%
}
	
\def\@xs@cutafteroccur@i#1#2#3{%
	\def\@xs@reserved@D{#1}\let\@xs@reserved@C\@empty\def\@xs@reserved@E{#2}%
	\decimalpart\z@\integerpart\z@\def\groupID{0}\let\@xs@nestlevel\z@
	\def\@xs@atendofgroup{%
		\expandafter\endgroup
		\expandafter\@xs@addtomacro\expandafter\@xs@reserved@C\expandafter{\expandafter{\@xs@reserved@C}}%
		\@xs@read@reserved@D}%
	\def\@xs@atoccurfound{%
		\ifnum\decimalpart=\numexpr(#3)\relax
			\global\let\@xs@reserved@D\@xs@reserved@D
			\global\let\@xs@reserved@C\@xs@reserved@C
			\global\let\groupID\groupID
			\@xs@exitallgroups
			\let\@xs@next\relax
		\else
			\expandafter\@xs@addtomacro\expandafter\@xs@reserved@C\expandafter{\@xs@reserved@E}%
			\let\@xs@next\@xs@read@reserved@D
		\fi}%
	\@xs@read@reserved@D
	\def\@xs@argument@A{#2}%
	\ifnum\decimalpart=\numexpr(#3)\relax 
		\let\@xs@reserved@E\@xs@reserved@D
		\expandafter\expandafter\expandafter\def\expandafter\expandafter\expandafter\@xs@reserved@D\expandafter\expandafter\expandafter{\expandafter\@xs@reserved@C\@xs@argument@A}%
	\else
		\let\@xs@reserved@C\@empty\let\@xs@reserved@E\@empty\let\groupID\@empty
	\fi
}

\@xs@newmacro*3\IfSubStr{1}{2}{0}{%
	\def\@xs@argument@A{#2}\def\@xs@argument@B{#3}%
	\expandafter\expandafter\expandafter\@xs@cutafteroccur
	\expandafter\expandafter\expandafter{\expandafter\@xs@argument@A\expandafter}\expandafter{\@xs@argument@B}{#1}%
	\expandafter\@xs@testempty\expandafter{\@xs@reserved@D}%
	\if@xs@empty
		\expandafter\@secondoftwo
	\else
		\expandafter\@firstoftwo
	\fi
}

\@xs@newmacro*2\IfBeginWith{}{2}{0}{%
	\def\@xs@argument@A{#1}\def\@xs@argument@B{#2}%
	\expandafter\@xs@testempty\expandafter{\@xs@argument@B}%
	\if@xs@empty
		\let\@xs@next\@secondoftwo
	\else
		\def\@xs@next{\expandafter\expandafter\expandafter\@xs@IfBeginWith@i
			\expandafter\expandafter\expandafter{\expandafter\@xs@argument@A\expandafter}\expandafter{\@xs@argument@B}}%
	\fi
	\@xs@next
}

\def\@xs@IfBeginWith@i#1#2{%
	\def\@xs@argument@A{#1}\def\@xs@argument@B{#2}%
	\expandafter\@xs@testempty\expandafter{\@xs@argument@B}%
	\if@xs@empty
		\let\@xs@next\@firstoftwo
	\else
		\expandafter\@xs@testempty\expandafter{\@xs@argument@A}
		\if@xs@empty
			\let\@xs@next\@secondoftwo
		\else
			\expandafter\@xs@returnfirstsyntaxunit\expandafter{\@xs@argument@B}\@xs@reserved@B
			\expandafter\@xs@returnfirstsyntaxunit\expandafter{\@xs@argument@A}\@xs@reserved@A
			\ifx\@xs@reserved@A\@xs@reserved@B
				\expandafter\@xs@removefirstsyntaxunit\expandafter{\@xs@argument@B}\@xs@reserved@B
				\expandafter\@xs@removefirstsyntaxunit\expandafter{\@xs@argument@A}\@xs@reserved@A
				\def\@xs@next{
					\expandafter\expandafter\expandafter\@xs@IfBeginWith@i
					\expandafter\expandafter\expandafter{\expandafter\@xs@reserved@A\expandafter}\expandafter{\@xs@reserved@B}}%
			\else
				\let\@xs@next\@secondoftwo
			\fi
		\fi
	\fi
	\@xs@next
}

\@xs@newmacro*2\IfEndWith{}{2}{0}{%
	\def\@xs@argument@A{#1}\def\@xs@argument@B{#2}%
	\@xs@testempty{#2}%
	\if@xs@empty
		\let\@xs@reserved@A\@secondoftwo
	\else
		\expandafter\expandafter\expandafter\@xs@StrCount
			\expandafter\expandafter\expandafter{\expandafter\@xs@argument@A\expandafter}\expandafter
			{\@xs@argument@B}[\@xs@reserved@A]%
		\ifnum\@xs@reserved@A=\z@
			\let\@xs@reserved@A\@secondoftwo
		\else
			\expandafter\@xs@testempty\expandafter{\@xs@reserved@C}%
			\if@xs@empty
				\let\@xs@reserved@A\@firstoftwo
			\else
				\let\@xs@reserved@A\@secondoftwo
			\fi
		\fi
	\fi
	\@xs@reserved@A
}

\@xs@newmacro*4\IfSubStrBefore{1,1}{3}{0}{\@xs@IfSubStrBefore@i[#1]{#2}{#3}{#4}}
\def\@xs@IfSubStrBefore@i[#1,#2]#3#4#5{%
	\def\@xs@reserved@C{#3}%
	\ifx\@xs@exploregroups\relax
		\let\@xs@reserved@B\@empty
		\let\@xs@nestlevel\z@
		\@xs@StrRemoveBraces@i
		\let\@xs@reserved@C\@xs@reserved@B
	\fi
	\def\@xs@reserved@A{#5}%
	\expandafter\expandafter\expandafter\@xs@cutafteroccur\expandafter\expandafter\expandafter{\expandafter\@xs@reserved@C\expandafter}\expandafter{\@xs@reserved@A}{#2}%
	\def\@xs@reserved@A{#4}%
	\expandafter\expandafter\expandafter\@xs@cutafteroccur\expandafter\expandafter\expandafter{\expandafter\@xs@reserved@C\expandafter}\expandafter{\@xs@reserved@A}{#1}%
	\let\groupID\@empty
	\expandafter\@xs@testempty\expandafter{\@xs@reserved@C}%
	\if@xs@empty
		\expandafter\@secondoftwo
	\else
		\expandafter\@firstoftwo
	\fi
}

\@xs@newmacro*4\IfSubStrBehind{1,1}{3}{0}{\@xs@IfSubStrBehind@i[#1]{#2}{#3}{#4}}
\def\@xs@IfSubStrBehind@i[#1,#2]#3#4#5{\@xs@IfSubStrBefore@i[#2,#1]{#3}{#5}{#4}}

\def\@xs@formatnumber#1#2{%
	\def\@xs@argument@A{#1}%
	\@xs@testempty{#1}%
	\if@xs@empty
		\def#2{0X}
	\else
		\@xs@returnfirstsyntaxunit{#1}\@xs@reserved@A
		\def\@xs@reserved@B{+}%
		\ifx\@xs@reserved@A\@xs@reserved@B
			\expandafter\@xs@removefirstsyntaxunit\expandafter{\@xs@argument@A}\@xs@reserved@C
			\expandafter\@xs@testempty\expandafter{\@xs@reserved@C}%
			\if@xs@empty
			 	\def#2{+0X}%
			 \else
			 	\expandafter\def\expandafter#2\expandafter{\expandafter+\expandafter0\@xs@reserved@C}%
			 \fi
		\else
			\def\@xs@reserved@B{-}%
			\ifx\@xs@reserved@A\@xs@reserved@B
				\expandafter\@xs@removefirstsyntaxunit\expandafter{\@xs@argument@A}\@xs@reserved@A
				\expandafter\@xs@testempty\expandafter{\@xs@reserved@A}%
				\if@xs@empty
				 	\def#2{-0X}%
				 \else
				 	\expandafter\def\expandafter#2\expandafter{\expandafter-\expandafter0\@xs@reserved@A}%
				 \fi
			\else
				\expandafter\def\expandafter#2\expandafter{\expandafter0\@xs@argument@A}%
			\fi
		\fi
	\fi
}

\@xs@newmacro\IfInteger{}{1}{0}{%
	\@xs@formatnumber{#1}\@xs@reserved@A
	\decimalpart\z@
	\afterassignment\@xs@defafterinteger\integerpart\@xs@reserved@A\relax\@xs@nil
	\let\@xs@after@intpart\@xs@afterinteger
	\expandafter\@xs@testdot\@xs@afterinteger\@xs@nil
	\ifx\@empty\@xs@afterdecimal
		\ifnum\decimalpart=\z@
			\let\@xs@next\@firstoftwo
		\else
			\let\@xs@afterinteger\@xs@after@intpart
			\let\@xs@next\@secondoftwo
		\fi
	\else
		\let\@xs@afterinteger\@xs@after@intpart
		\let\@xs@next\@secondoftwo
	\fi
	\@xs@next
}

\@xs@newmacro\IfDecimal{}{1}{0}{%
	\@xs@formatnumber{#1}\@xs@reserved@A
	\decimalpart\z@
	\afterassignment\@xs@defafterinteger\integerpart\@xs@reserved@A\relax\@xs@nil
	\expandafter\@xs@testdot\@xs@afterinteger\@xs@nil
	\ifx\@empty\@xs@afterdecimal
		\expandafter\@firstoftwo
	\else
		\expandafter\@secondoftwo
	\fi
}

\def\@xs@defafterinteger#1\relax\@xs@nil{\def\@xs@afterinteger{#1}}
\def\@xs@testdot{%
	\let\xs@decsep\@empty
	\@ifnextchar.%
		{\def\xs@decsep{.}\@xs@readdecimalpart}%
		{\@xs@testcomma}%
}
\def\@xs@testcomma{%
	\@ifnextchar,%
		{\def\xs@dessep{,}\@xs@readdecimalpart}%
		{\@xs@endnumber}%
}
\def\@xs@readdecimalpart#1#2\@xs@nil{%
	\ifx\@empty#2\@empty
		\def\@xs@reserved@A{0X}%
	\else
		\def\@xs@reserved@A{0#2}%
	\fi
	\afterassignment\@xs@defafterinteger\decimalpart\@xs@reserved@A\relax\@xs@nil
	\expandafter\@xs@endnumber\@xs@afterinteger\@xs@nil
}
\def\@xs@endnumber#1\@xs@nil{\def\@xs@afterdecimal{#1}}

\def\@xs@IfStrEqFalse@i#1#2{\let\@xs@reserved@A\@secondoftwo}
\def\@xs@IfStrEqFalse@ii#1#2{
	\@xs@IfDecimal{#1}%
		{\@xs@IfDecimal{#2}%
			{\ifdim#1pt=#2pt
				\let\@xs@reserved@A\@firstoftwo
			\else
				\let\@xs@reserved@A\@secondoftwo
			\fi
			}%
			{\let\@xs@reserved@A\@secondoftwo}
		}%
		{\let\@xs@reserved@A\@secondoftwo}
}

\def\@xs@TestEqual#1#2{
	\def\@xs@reserved@A{#1}\def\@xs@reserved@B{#2}%
	\ifx\@xs@reserved@A\@xs@reserved@B
		\let\@xs@reserved@A\@firstoftwo
	\else
		\expandafter\expandafter\expandafter\@xs@reserved@D\expandafter\expandafter\expandafter{\expandafter\@xs@reserved@A\expandafter}\expandafter{\@xs@reserved@B}%
	\fi
	\@xs@reserved@A
}

\@xs@newmacro*2\IfStrEq{}{2}{0}{
	\let\@xs@reserved@D\@xs@IfStrEqFalse@i
	\@xs@TestEqual{#1}{#2}%
}

\@xs@newmacro*2\IfEq{}{2}{0}{
	\let\@xs@reserved@D\@xs@IfStrEqFalse@ii
	\@xs@TestEqual{#1}{#2}%
}

\def\IfStrEqCase{%
	\@ifstar
		{\def\@xs@reserved@E{\IfStrEq*}\@xs@IfStrCase}%
		{\def\@xs@reserved@E{\IfStrEq}\@xs@IfStrCase}%
}
\def\@xs@IfStrCase#1#2{\@testopt{\@xs@IfStringCase{#1}{#2}}{}}

\def\IfEqCase{%
	\@ifstar
		{\def\@xs@reserved@E{\IfEq*}\@xs@IfEqCase}%
		{\def\@xs@reserved@E{\IfEq}\@xs@IfEqCase}%
}
\def\@xs@IfEqCase#1#2{\@testopt{\@xs@IfStringCase{#1}{#2}}{}}

\def\@xs@IfStringCase#1#2[#3]{%
	\def\@xs@testcase##1##2##3\@xs@nil{
		\@xs@reserved@E{#1}{##1}%
			{##2}
			{\@xs@testempty{##3}%
			 \if@xs@empty
			 	\def\@xs@next{#3}
			 \else
			 	\def\@xs@next{\@xs@testcase##3\@xs@nil}
			 \fi
			 \@xs@next
			 }%
	}%
	\@xs@testcase#2\@xs@nil
}

\@xs@newmacro*3\StrBefore{1}{2}{1}{%
	\@xs@cutafteroccur{#2}{#3}{#1}%
	\expandafter\@xs@ReturnResult\expandafter{\@xs@reserved@C}{#4}%
}

\@xs@newmacro*3\StrBehind{1}{2}{1}{%
	\@xs@cutafteroccur{#2}{#3}{#1}%
	\expandafter\@xs@ReturnResult\expandafter{\@xs@reserved@E}{#4}%
}

\@xs@newmacro*4\StrBetween{1,1}{3}{1}{\@xs@StrBetween@i[#1]{#2}{#3}{#4}[#5]}
\def\@xs@StrBetween@i[#1,#2]#3#4#5[#6]{%
	\begingroup
		\noexploregroups
		\@xs@cutafteroccur{#3}{#5}{#2}%
		\expandafter\@xs@cutafteroccur\expandafter{\@xs@reserved@C}{#4}{#1}%
		\expandafter
	\endgroup
	\expandafter\@xs@ReturnResult\expandafter{\@xs@reserved@E}{#6}%
	\let\groupID\@empty
}

\def\exploregroups{\let\@xs@exploregroups\relax}
\def\noexploregroups{\def\@xs@exploregroups{\let\@xs@toks0\relax}}
\def\saveexploremode{\let\@xs@saveexploremode\@xs@exploregroups}
\def\restoreexploremode{\let\@xs@exploregroups\@xs@saveexploremode}

\@xs@newmacro\StrSubstitute{0}{3}{1}{%
	\def\@xs@reserved@D{#2}\let\@xs@reserved@C\@empty\def\@xs@reserved@E{#3}%
	\def\@xs@argument@C{#3}\def\@xs@argument@D{#4}%
	\decimalpart\z@\integerpart\z@\def\groupID{0}\let\@xs@nestlevel\z@
	\def\@xs@atendofgroup{%
		\expandafter\endgroup
		\expandafter\@xs@addtomacro\expandafter\@xs@reserved@C\expandafter{\expandafter{\@xs@reserved@C}}%
		\@xs@read@reserved@D
	}%
	\def\@xs@atoccurfound{%
		\ifnum#1<\@ne
			\let\@xs@reserved@A\@xs@argument@D
		\else
			\ifnum\decimalpart>#1
				\let\@xs@reserved@A\@xs@argument@C
			\else
				\let\@xs@reserved@A\@xs@argument@D
			\fi
		\fi
		\expandafter\@xs@addtomacro\expandafter\@xs@reserved@C\expandafter{\@xs@reserved@A}%
		\@xs@read@reserved@D
	}%
	\@xs@testempty{#3}%
	\if@xs@empty
		\expandafter\@xs@ReturnResult\expandafter{\@xs@reserved@D}{#5}%
	\else
		\@xs@read@reserved@D
		\expandafter\@xs@ReturnResult\expandafter{\@xs@reserved@C}{#5}%
	\fi
}

\@xs@newmacro\StrDel{0}{2}{1}{\@xs@StrSubstitute[#1]{#2}{#3}{}[#4]}

\def\@xs@exitallgroups{\ifnum\@xs@nestlevel>\z@\endgroup\expandafter\@xs@exitallgroups\fi}

\@xs@newmacro\StrLen{}{1}{1}{%
	\def\@xs@reserved@C{#1}%
	\decimalpart\z@
	\let\@xs@nestlevel\z@
	\def\groupID{0}%
	\let\@xs@atbegingroup\relax
	\def\@xs@atendofgroup{\endgroup\@xs@read@reserved@C}%
	\let\@xs@atnextsyntaxunit\relax
	\@xs@read@reserved@C
	\expandafter\@xs@ReturnResult\expandafter{\number\decimalpart}{#2}%
}

\def\@xs@continuetonext{%
	\expandafter\@xs@testempty\expandafter{\@xs@reserved@C}%
	\if@xs@empty
		\ifnum\@xs@nestlevel>\z@
			\expandafter\endgroup\expandafter\@xs@addtomacro\expandafter\@xs@reserved@B\expandafter{\expandafter{\@xs@reserved@B}}
			\expandafter\expandafter\expandafter\@xs@continuetonext
		\fi
	\fi
}%

\def\@xs@manage@groupID{%
	\begingroup\def\@xs@reserved@A{0}%
	\ifx\@xs@reserved@A\groupID
		\endgroup\edef\groupID{\number\integerpart}
	\else
		\endgroup\expandafter\@xs@addtomacro\expandafter\groupID\expandafter{\expandafter,\number\integerpart}%
	\fi
}

\def\StrSplit{%
	\@ifstar
		{\let\@xs@reserved@E\@xs@continuetonext\StrSpl@t}%
		{\let\@xs@reserved@E\relax\StrSpl@t}%
}
\@xs@newmacro\StrSpl@t{}{2}{0}{\@xs@StrSplit@i{#2}{#1}\@xs@StrSplit@ii}
\def\@xs@StrSplit@i#1#2{%
	\def\@xs@reserved@D{#1}\def\@xs@reserved@C{#2}\let\@xs@reserved@B\@empty\let\groupID\@empty
	\ifnum#1>\z@
		\decimalpart\z@\integerpart\z@\def\groupID{0}\let\@xs@nestlevel\z@
		\def\@xs@atendofgroup{%
			\expandafter\endgroup
			\expandafter\@xs@addtomacro\expandafter\@xs@reserved@B\expandafter{\expandafter{\@xs@reserved@B}}%
			\@xs@read@reserved@C
		}%
		\def\@xs@atbegingroup{\let\@xs@reserved@B\@empty}%
		\def\@xs@atnextsyntaxunit{%
			\expandafter\@xs@addtomacro\expandafter\@xs@reserved@B\expandafter{\@xs@reserved@A}%
			\ifnum\decimalpart=\@xs@reserved@D\relax
				\ifx\@xs@reserved@C\@empty\@xs@reserved@E\fi
				\global\let\@xs@reserved@B\@xs@reserved@B
				\global\let\@xs@reserved@C\@xs@reserved@C
				\global\let\groupID\groupID
				\@xs@exitallgroups
				\let\@xs@next\relax
			\fi
		}%
		\@xs@read@reserved@C
	\fi
}
\def\@xs@StrSplit@ii#1#2{\let#1\@xs@reserved@B\let#2\@xs@reserved@C}

\@xs@newmacro*3\StrCut{1}{2}{0}{%
	\@xs@testempty{#3}%
	\if@xs@empty\expandafter\@firstoftwo\else\expandafter\@secondoftwo\fi
		{\let\groupID\@empty
		\let\@xs@reserved@C\@empty
		\let\@xs@reserved@E\@empty
		}
		{\ifnum#1<\@ne\expandafter\@firstoftwo\else\expandafter\@secondoftwo\fi
			{\@xs@StrCut@ii{#2}{#3}1}
			{\@xs@StrCut@ii{#2}{#3}{#1}}%
		}%
	\@xs@StrCut@iii
}

\def\@xs@StrCut@ii#1#2#3{%
	\def\@xs@reserved@D{#1}%
	\let\@xs@reserved@C\@empty
	\def\@xs@reserved@E{#2}%
	\decimalpart\z@\integerpart\z@
	\def\groupID{0}%
	\let\@xs@nestlevel\z@
	\def\@xs@atendofgroup{%
		\expandafter\endgroup
		\expandafter\@xs@addtomacro\expandafter\@xs@reserved@C\expandafter{\expandafter{\@xs@reserved@C}}%
		\@xs@read@reserved@D
	}%
	\def\@xs@atoccurfound{%
		\ifnum\decimalpart=\numexpr(#3)\relax
			\global\let\@xs@reserved@D\@xs@reserved@D
			\global\let\@xs@reserved@C\@xs@reserved@C
			\global\let\groupID\groupID
			\@xs@exitallgroups
			\let\@xs@next\relax
		\else
				\expandafter\@xs@addtomacro\expandafter\@xs@reserved@C\expandafter{\@xs@reserved@E}%
			\let\@xs@next\@xs@read@reserved@D
		\fi
	}%
	\@xs@read@reserved@D
	\def\@xs@argument@A{#2}%
	\let\@xs@reserved@E\@xs@reserved@D
		\expandafter\expandafter\expandafter
	\def
		\expandafter\expandafter\expandafter
	\@xs@reserved@D
		\expandafter\expandafter\expandafter
	{\expandafter\@xs@reserved@C\@xs@argument@A}%
}

\def\@xs@StrCut@iii#1#2{\let#1\@xs@reserved@C\let#2\@xs@reserved@E}

\@xs@newmacro\StrMid{}{3}{1}{%
	\begingroup
		\noexploregroups
		\let\@xs@reserved@E\relax
		\@xs@StrSplit@i{#3}{#1}%
		\edef\@xs@reserved@C{\number\numexpr#2-1}%
		\let\@xs@reserved@E\relax
		\expandafter\expandafter\expandafter\@xs@StrSplit@i\expandafter\expandafter\expandafter{\expandafter\@xs@reserved@C\expandafter}\expandafter{\@xs@reserved@B}%
	\expandafter\endgroup
	\expandafter\@xs@ReturnResult\expandafter{\@xs@reserved@C}{#4}%
	\let\groupID\@empty
}

\@xs@newmacro\StrGobbleLeft{}{2}{1}{%
	\let\@xs@reserved@E\relax
	\@xs@StrSplit@i{#2}{#1}%
	\expandafter\@xs@ReturnResult\expandafter{\@xs@reserved@C}{#3}%
}

\@xs@newmacro\StrLeft{}{2}{1}{%
	\let\@xs@reserved@E\relax
	\@xs@StrSplit@i{#2}{#1}%
	\expandafter\@xs@ReturnResult\expandafter{\@xs@reserved@B}{#3}%
}

\@xs@newmacro\StrGobbleRight{}{2}{1}{%
	\@xs@StrLen{#1}[\@xs@reserved@D]%
	\let\@xs@reserved@E\relax
	\expandafter\@xs@StrSplit@i\expandafter{\number\numexpr\@xs@reserved@D-#2}{#1}%
	\expandafter\@xs@ReturnResult\expandafter{\@xs@reserved@B}{#3}%
}

\@xs@newmacro\StrRight{}{2}{1}{%
	\@xs@StrLen{#1}[\@xs@reserved@D]%
	\let\@xs@reserved@E\relax
	\expandafter\@xs@StrSplit@i\expandafter{\number\numexpr\@xs@reserved@D-#2}{#1}%
	\expandafter\@xs@ReturnResult\expandafter{\@xs@reserved@C}{#3}%
}

\@xs@newmacro\StrChar{}{2}{1}{%
	\let\@xs@reserved@B\@empty
	\def\@xs@reserved@C{#1}\def\@xs@reserved@D{#2}%
	\ifnum#2>\z@
		\def\groupID{0}\let\@xs@nestlevel\z@\integerpart\z@\decimalpart\z@
		\let\@xs@atbegingroup\relax
		\def\@xs@atendofgroup{\endgroup\@xs@read@reserved@C}%
		\def\@xs@atnextsyntaxunit{%
			\ifnum\decimalpart=\@xs@reserved@D
				\global\let\@xs@reserved@B\@xs@reserved@A
				\global\let\groupID\groupID
				\@xs@exitallgroups
				\let\@xs@next\relax
			\fi
		}%
		\@xs@read@reserved@C
	\fi
	\expandafter\@xs@testempty\expandafter{\@xs@reserved@B}%
	\if@xs@empty\let\groupID\@empty\fi
	\expandafter\@xs@ReturnResult\expandafter{\@xs@reserved@B}{#3}%
}

\@xs@newmacro\StrCount{}{2}{1}{%
	\@xs@testempty{#2}%
	\def\@xs@reserved@D{#1}\def\@xs@reserved@E{#2}\let\@xs@reserved@C\@empty
	\if@xs@empty
		\@xs@ReturnResult{0}{#3}%
	\else
		\decimalpart\z@\integerpart\z@\def\groupID{0}\let\@xs@nestlevel\z@
		\def\@xs@atendofgroup{%
			\expandafter\endgroup
			\expandafter\@xs@addtomacro\expandafter\@xs@reserved@C\expandafter{\expandafter{\@xs@reserved@C}}%
			\@xs@read@reserved@D
		}%
		\def\@xs@atoccurfound{\let\@xs@reserved@C\@empty\@xs@read@reserved@D}%
		\@xs@read@reserved@D
		\expandafter\@xs@ReturnResult\expandafter{\number\decimalpart}{#3}%
	\fi
}

\@xs@newmacro\StrPosition{1}{2}{1}{%
	\@xs@cutafteroccur{#2}{#3}{#1}%
	\let\@xs@reserved@E\groupID
	\ifx\@xs@reserved@C\@xs@reserved@D
		\@xs@ReturnResult{0}{#4}%
		\let\@xs@reserved@E\@empty
	\else
		\expandafter\@xs@StrLen\expandafter{\@xs@reserved@C}[\@xs@reserved@C]%
		\expandafter\@xs@ReturnResult\expandafter{\number\numexpr\@xs@reserved@C+1}{#4}%
	\fi
	\let\groupID\@xs@reserved@E
}

\def\comparestrict{\let\@xs@comparecoeff\@ne}
\def\comparenormal{\let\@xs@comparecoeff\z@}
\def\savecomparemode{\let\@xs@saved@comparecoeff\@xs@comparecoeff}
\def\restorecomparemode{\let\@xs@comparecoeff\@xs@saved@comparecoeff}
\@xs@newmacro*2\StrCompare{}{2}{1}{%
	\def\@xs@reserved@A{#1}%
	\def\@xs@reserved@B{#2}%
	\ifx\@xs@reserved@B\@xs@reserved@A
		\@xs@ReturnResult{0}{#3}%
	\else
		\def\@xs@next{\@xs@StrCompare@i{#1}{#2}{#3}}%
		\expandafter\@xs@next
	\fi
}

\def\@xs@StrCompare@i#1#2#3{%
	\def\@xs@StrCompare@iii##1{%
		\let\@xs@reserved@A\@empty
		\expandafter\@xs@testempty\expandafter{\@xs@reserved@C}%
		\if@xs@empty
			\def\@xs@reserved@A{*\@xs@comparecoeff}%
		\else
			\expandafter\@xs@testempty\expandafter{\@xs@reserved@D}%
			\if@xs@empty
				\def\@xs@reserved@A{*\@xs@comparecoeff}%
			\fi
		\fi
		\def\@xs@next{%
			\expandafter\@xs@ReturnResult\expandafter
			{\number\numexpr##1\@xs@reserved@A\relax}{#3}%
		}%
	}%
	\def\@xs@StrCompare@ii##1{
		\expandafter\@xs@returnfirstsyntaxunit\expandafter{\@xs@reserved@C}\@xs@reserved@A
		\expandafter\@xs@returnfirstsyntaxunit\expandafter{\@xs@reserved@D}\@xs@reserved@B
		\ifx\@xs@reserved@B\@xs@reserved@A
			\expandafter\@xs@testempty\expandafter{\@xs@reserved@A}%
			\if@xs@empty
				\@xs@StrCompare@iii{##1}
			\else
				\def\@xs@next{\expandafter\@xs@StrCompare@ii\expandafter{\number\numexpr##1+1}}
				\expandafter\@xs@removefirstsyntaxunit\expandafter{\@xs@reserved@C}\@xs@reserved@C
				\expandafter\@xs@removefirstsyntaxunit\expandafter{\@xs@reserved@D}\@xs@reserved@D
			\fi
		\else
			\@xs@StrCompare@iii{##1}%
		\fi
		\@xs@next
	}%
	\def\@xs@reserved@C{#1}\def\@xs@reserved@D{#2}%
	\@xs@StrCompare@ii1%
}

\@xs@newmacro\StrFindGroup{}{2}{1}{%
	\def\@xs@reserved@A{#2}\def\@xs@reserved@B{0}%
	\ifx\@xs@reserved@A\@xs@reserved@B
		\def\@xs@next{\@xs@ReturnResult{{#1}}{#3}}%
	\else
		\def\@xs@next{\@xs@StrFindGroup@i{#1}{#2}[#3]}%
	\fi
	\@xs@next
}

\def\@xs@StrFindGroup@i#1#2[#3]{%
	\def\@xs@StrFindGroup@ii{%
		\expandafter\@xs@testempty\expandafter{\@xs@reserved@C}%
		\if@xs@empty
			\def\@xs@next{\@xs@ReturnResult{}{#3}}
		\else
			\expandafter\@xs@returnfirstsyntaxunit\expandafter{\@xs@reserved@C}\@xs@reserved@D
			\ifx\bgroup\@xs@toks
				\advance\decimalpart\@ne
				\ifnum\decimalpart=\@xs@reserved@A
					\ifx\@empty\@xs@reserved@B
						\def\@xs@next{\expandafter\@xs@ReturnResult\expandafter{\@xs@reserved@D}{#3}}
					\else
						\expandafter\def\expandafter\@xs@next\expandafter{\expandafter\@xs@StrFindGroup@i\@xs@reserved@D}
						\expandafter\@xs@addtomacro\expandafter\@xs@next\expandafter{\expandafter{\@xs@reserved@B}[#3]}
					\fi
				\else
					\expandafter\@xs@removefirstsyntaxunit\expandafter{\@xs@reserved@C}\@xs@reserved@C
					\let\@xs@next\@xs@StrFindGroup@ii
				\fi
			\else
				\expandafter\@xs@removefirstsyntaxunit\expandafter{\@xs@reserved@C}\@xs@reserved@C
				\let\@xs@next\@xs@StrFindGroup@ii
			\fi
		\fi
		\@xs@next
	}%
	\@xs@extractgroupnumber{#2}\@xs@reserved@A\@xs@reserved@B
	\decimalpart\z@
	\ifnum\@xs@reserved@A>\z@\def\@xs@reserved@C{#1}\else\let\@xs@reserved@C\@empty\fi
	\@xs@StrFindGroup@ii
}

\def\@xs@extractgroupnumber#1#2#3{%
	\def\@xs@extractgroupnumber@i##1,##2\@xs@nil{\def#2{##1}\def#3{##2}}%
	\@xs@extractgroupnumber@i#1,\@xs@nil
	\ifx\@empty#3\else\@xs@extractgroupnumber@i#1\@xs@nil\fi
}

\def\expandingroups{\let\@xs@expandingroups\exploregroups}
\def\noexpandingroups{\let\@xs@expandingroups\noexploregroups}
\def\StrExpand{\@testopt{\@xs@StrExpand}{1}}
\def\@xs@StrExpand[#1]#2#3{%
	\begingroup
		\@xs@expandingroups
		\ifnum#1>\z@
			\integerpart#1\relax
			\decimalpart\z@\def\groupID{0}\let\@xs@nestlevel\z@
			\def\@xs@atendofgroup{%
				\expandafter\endgroup
				\expandafter\@xs@addtomacro\expandafter\@xs@reserved@B\expandafter{\expandafter{\@xs@reserved@B}}%
				\@xs@read@reserved@C
			}%
			\def\@xs@atbegingroup{\let\@xs@reserved@B\@empty}%
			\def\@xs@atnextsyntaxunit{%
				\expandafter\expandafter\expandafter\@xs@addtomacro
				\expandafter\expandafter\expandafter\@xs@reserved@B
				\expandafter\expandafter\expandafter{\@xs@reserved@A}%
			}%
			\def\@xs@reserved@C{#2}%
			\@xs@StrExpand@i{#1}
		\else
			\def\@xs@reserved@B{#2}%
		\fi
		\global\let\@xs@reserved@B\@xs@reserved@B
	\endgroup
	\let#3\@xs@reserved@B
	\let\groupID\@empty
}

\def\@xs@StrExpand@i#1{%
	\ifnum#1>\z@
		\let\@xs@reserved@B\@empty
		\@xs@read@reserved@C
		\let\@xs@reserved@C\@xs@reserved@B
		\def\@xs@reserved@A{\expandafter\@xs@StrExpand@i\expandafter{\number\numexpr#1-1}}%
	\else
		\let\@xs@reserved@A\relax
	\fi
	\@xs@reserved@A
}

\def\scancs{\@testopt{\@xs@scancs}{1}}
\def\@xs@scancs[#1]#2#3{%
	\@xs@StrExpand[#1]{#3}{#2}%
	\edef#2{\detokenize\expandafter{#2}}%
}

\catcode`\@=\CurrentAtCatcode\relax
\setverbdelim{|}%
\fullexpandarg\saveexpandmode
\comparenormal\savecomparemode
\noexploregroups\saveexploremode
\expandingroups

\newcommand\myscale{0.8}

\newcommand\tcirc[3]{
	\ifthenelse{\equal{#1}{w}}{\filldraw[fill=white,draw=black] (#2) circle (0.08);}{}%
	\ifthenelse{\equal{#1}{b}}{\filldraw[black] (#2) circle (0.08);}{}%
	\draw (#2) node[above=2pt] {#3};
	}

\newcommand\tdots[1]{\draw (#1) ++(0.55,0) node {$\cdots$}}

\newcommand\bond[1]{\draw (#1) -- +(1,0)}

\newcommand\vbond[1]{\draw (#1) -- +(0,-1)}

\newcommand\diagbond[2]{
	\ifthenelse{\equal{#1}{u}}{
		\draw[semithick] (#2) -- +(0.5,0.865);
	}{}
	\ifthenelse{\equal{#1}{d}}{
		\draw[semithick] (#2) -- +(0.5,-0.865);
	}{}
	}

\newcommand\dbond[2]{
	\draw (#2) ++(0.03,0.03) -- +(0.94,0);
	\draw (#2) ++(0.03,-0.03) -- +(0.94,0);
	\ifthenelse{\equal{#1}{r}}{
		\draw[semithick] (#2) ++(0.6,0) ++(-0.15,0.2) -- ++(0.15,-0.2) -- +(-0.15,-0.2);
	}{}
	\ifthenelse{\equal{#1}{l}}{
		\draw[semithick] (#2) ++(0.45,0) ++(0.15,0.2) -- ++(-0.15,-0.2) -- +(0.15,-0.2);
	}{}
	}

\newcommand\tbond[2]{
	\draw (#2)  -- +(1,0);
	\draw (#2) ++(0.05,0.06) -- +(0.9,0);
	\draw (#2) ++(0.05,-0.06) -- +(0.9,0);
	\ifthenelse{\equal{#1}{r}}{
		\draw[semithick] (#2) ++(0.6,0) ++(-0.15,0.2) -- ++(0.15,-0.2) -- +(-0.15,-0.2);
	}{}
	\ifthenelse{\equal{#1}{l}}{
		\draw[semithick] (#2) ++(0.45,0) ++(0.15,0.2) -- ++(-0.15,-0.2) -- +(0.15,-0.2);
	}{}
	}

\newcommand\tcross[2]{
	\draw (#1) node[above=2pt] {#2};
	\draw (#1) ++(-0.12,-0.12)-- +(0.24, 0.24);
	\draw (#1) ++(-0.12, 0.12)-- +(0.24,-0.24);
	}

\newcommand\tsquare[2]{
		\draw[semithick,color=blue] (#1) ++(-0.15,-0.15) rectangle ++(0.3,0.3);
		\tcross{#1}{#2};
		}

\newcommand\tstar[2]{
	\draw[color=red] (#1) node {\Large$*$};
	\draw (#1) node[above=2pt] {#2};
	}

\newcommand\DDnode[3]{
\ifthenelse{\equal{#1}{w}}{\tcirc{w}{#2}{#3}}{}		
\ifthenelse{\equal{#1}{b}}{\tcirc{b}{#2}{#3}}{}		
\ifthenelse{\equal{#1}{x}}{\tcross{#2}{#3}}{}		
\ifthenelse{\equal{#1}{s}}{\tstar{#2}{#3}}{}		
\ifthenelse{\equal{#1}{q}}{\tsquare{#2}{#3}}{}		
}

\newcommand\Edd[2]{
 \begin{tiny}
 \begin{tikzpicture}[scale=\myscale,baseline=-3pt]
 \foreach \x in {0,1,2,3} {
	\bond{\x,0};
 }
 \vbond{2,0};

 \StrLen{#1}[\Ernk]

 \StrChar{#1}{1}[\nodetype];
 \DDnode{\nodetype}{0,0}{\StrBefore{#2}{,}};
 \StrChar{#1}{2}[\nodetype];
 \DDnode{\nodetype}{2,-1}{\StrBetween[1,2]{#2}{,}{,}};
 \StrChar{#1}{3}[\nodetype];
 \DDnode{\nodetype}{1,0}{\StrBetween[2,3]{#2}{,}{,}};
 \StrChar{#1}{4}[\nodetype];
 \DDnode{\nodetype}{2,0}{\StrBetween[3,4]{#2}{,}{,}};
 \StrChar{#1}{5}[\nodetype];
 \DDnode{\nodetype}{3,0}{\StrBetween[4,5]{#2}{,}{,}};
 \StrChar{#1}{6}[\nodetype];

 \ifthenelse{\equal{\Ernk}{6}}{
 		\DDnode{\nodetype}{4,0}{\StrBehind[5]{#2}{,}};
 		\useasboundingbox (-.4,-1.2) rectangle (4.4,0.55);
	}{}%

 \ifthenelse{\equal{\Ernk}{7}}{
 		\bond{4,0};
 		\DDnode{\nodetype}{4,0}{\StrBetween[5,6]{#2}{,}{,}};
		\StrChar{#1}{7}[\nodetype];
		\DDnode{\nodetype}{5,0}{\StrBehind[6]{#2}{,}};
 		\useasboundingbox (-.4,-1.2) rectangle (5.4,0.55);
	}{}%

 \ifthenelse{\equal{\Ernk}{8}}{
 		\bond{4,0};
 		\bond{5,0};
 		\DDnode{\nodetype}{4,0}{\StrBetween[5,6]{#2}{,}{,}};
		\StrChar{#1}{7}[\nodetype];
		\DDnode{\nodetype}{5,0}{\StrBetween[6,7]{#2}{,}{,}};
		\StrChar{#1}{8}[\nodetype];
		\DDnode{\nodetype}{6,0}{\StrBehind[7]{#2}{,}};
		\useasboundingbox (-.4,-1.2) rectangle (6.4,0.55);
	}{}%

 \end{tikzpicture}
 \end{tiny}
 }

 \newcommand\CPtwo[2]{
 \begin{tikzpicture}[scale=\myscale,baseline=-3pt]

 \tikzstyle{every node}=[font=\tiny]
 \newcommand\myb{0.3}

 \bond{0,\myb};

 \draw (0,-\myb) node[above=-1pt] {$\updownarrow$};
 \draw (1,-\myb) node[above=-1pt] {$\updownarrow$};

 \bond{0,-\myb};

 \StrBefore{#1}{,}[\labelone]
 \StrBehind[1]{#1}{,}[\labeltwo]

 \DDnode{x}{0,\myb}{\labelone};
 \DDnode{w}{1,\myb}{\labeltwo};

 \StrBefore{#2}{,}[\labelone]
 \StrBehind[1]{#2}{,}[\labeltwo]

 \DDnode{x}{0,-\myb}{};
 \DDnode{w}{1,-\myb}{};

 \draw (0,-\myb) node[below=2pt] {\labelone};
 \draw (1,-\myb) node[below=2pt] {\labeltwo};

 \end{tikzpicture}
 }
 \newcommand\CPthree[2]{
 \begin{tikzpicture}[scale=\myscale,baseline=-3pt]

 \tikzstyle{every node}=[font=\tiny]
 \newcommand\myb{0.3}
 \bond{0,\myb};
 \bond{1,\myb};

 \draw (0,-\myb) node[above=-1pt] {$\updownarrow$};
 \draw (1,-\myb) node[above=-1pt] {$\updownarrow$};
 \draw (2,-\myb) node[above=-1pt] {$\updownarrow$};

 \bond{0,-\myb};
 \bond{1,-\myb};

 \StrBefore{#1}{,}[\labelone]
 \StrBetween[1,2]{#1}{,}{,}[\labeltwo]
 \StrBehind[2]{#1}{,}[\labelthree]

 \DDnode{x}{0,\myb}{\labelone};
 \DDnode{w}{1,\myb}{\labeltwo};
 \DDnode{w}{2,\myb}{\labelthree};

 \StrBefore{#2}{,}[\labelone]
 \StrBetween[1,2]{#2}{,}{,}[\labeltwo]
 \StrBehind[2]{#2}{,}[\labelthree]

 \DDnode{x}{0,-\myb}{};
 \DDnode{w}{1,-\myb}{};
 \DDnode{w}{2,-\myb}{};

 \draw (0,-\myb) node[below=2pt] {\labelone};
 \draw (1,-\myb) node[below=2pt] {\labeltwo};
 \draw (2,-\myb) node[below=2pt] {\labelthree};

 \end{tikzpicture}
 }

 \newcommand\SLCR[2]{
 \begin{tikzpicture}[scale=\myscale,baseline=-3pt]

 \tikzstyle{every node}=[font=\tiny]
 \newcommand\myb{0.3}
 \bond{0,\myb};
 \bond{1,\myb};
 \bond{2,\myb};
 \tdots{3,\myb};
 \bond{4,\myb};

 \draw (0,-\myb) node[above=-1pt] {$\updownarrow$};
 \draw (1,-\myb) node[above=-1pt] {$\updownarrow$};
 \draw (2,-\myb) node[above=-1pt] {$\updownarrow$};
 \draw (3,-\myb) node[above=-1pt] {$\updownarrow$};
 \draw (4,-\myb) node[above=-1pt] {$\updownarrow$};
 \draw (5,-\myb) node[above=-1pt] {$\updownarrow$};

 \bond{0,-\myb};
 \bond{1,-\myb};
 \bond{2,-\myb};
 \tdots{3,-\myb};
 \bond{4,-\myb};

 \StrBefore{#1}{,}[\labelone]
 \StrBetween[1,2]{#1}{,}{,}[\labeltwo]
 \StrBetween[2,3]{#1}{,}{,}[\labelthree]
 \StrBetween[3,4]{#1}{,}{,}[\labelfour]
 \StrBetween[4,5]{#1}{,}{,}[\labelfive]
 \StrBehind[5]{#1}{,}[\labelsix]

 \DDnode{w}{0,\myb}{\labelone};
 \DDnode{w}{1,\myb}{\labeltwo};
 \DDnode{w}{2,\myb}{\labelthree};
 \DDnode{w}{3,\myb}{\labelfour};
 \DDnode{w}{4,\myb}{\labelfive};
 \DDnode{w}{5,\myb}{\labelsix};

 \StrBefore{#2}{,}[\labelone]
 \StrBetween[1,2]{#2}{,}{,}[\labeltwo]
 \StrBetween[2,3]{#2}{,}{,}[\labelthree]
 \StrBetween[3,4]{#2}{,}{,}[\labelfour]
 \StrBetween[4,5]{#2}{,}{,}[\labelfive]
 \StrBehind[5]{#2}{,}[\labelsix]

 \DDnode{w}{0,-\myb}{};
 \DDnode{w}{1,-\myb}{};
 \DDnode{w}{2,-\myb}{};
 \DDnode{w}{3,-\myb}{};
 \DDnode{w}{4,-\myb}{};
 \DDnode{w}{5,-\myb}{};

 \draw (0,-\myb) node[below=2pt] {\labelone};
 \draw (1,-\myb) node[below=2pt] {\labeltwo};
 \draw (2,-\myb) node[below=2pt] {\labelthree};
 \draw (3,-\myb) node[below=2pt] {\labelfour};
 \draw (4,-\myb) node[below=2pt] {\labelfive};
 \draw (5,-\myb) node[below=2pt] {\labelsix};

 \end{tikzpicture}
 }

 \newcommand\CPgen[2]{
 \begin{tikzpicture}[scale=\myscale,baseline=-3pt]

 \tikzstyle{every node}=[font=\tiny]
 \newcommand\myb{0.3}
 \bond{0,\myb};
 \bond{1,\myb};
 \bond{2,\myb};
 \tdots{3,\myb};
 \bond{4,\myb};

 \draw (0,-\myb) node[above=-1pt] {$\updownarrow$};
 \draw (1,-\myb) node[above=-1pt] {$\updownarrow$};
 \draw (2,-\myb) node[above=-1pt] {$\updownarrow$};
 \draw (3,-\myb) node[above=-1pt] {$\updownarrow$};
 \draw (4,-\myb) node[above=-1pt] {$\updownarrow$};
 \draw (5,-\myb) node[above=-1pt] {$\updownarrow$};

 \bond{0,-\myb};
 \bond{1,-\myb};
 \bond{2,-\myb};
 \tdots{3,-\myb};
 \bond{4,-\myb};

 \StrBefore{#1}{,}[\labelone]
 \StrBetween[1,2]{#1}{,}{,}[\labeltwo]
 \StrBetween[2,3]{#1}{,}{,}[\labelthree]
 \StrBetween[3,4]{#1}{,}{,}[\labelfour]
 \StrBetween[4,5]{#1}{,}{,}[\labelfive]
 \StrBehind[5]{#1}{,}[\labelsix]

 \DDnode{x}{0,\myb}{\labelone};
 \DDnode{w}{1,\myb}{\labeltwo};
 \DDnode{w}{2,\myb}{\labelthree};
 \DDnode{w}{3,\myb}{\labelfour};
 \DDnode{w}{4,\myb}{\labelfive};
 \DDnode{w}{5,\myb}{\labelsix};

 \StrBefore{#2}{,}[\labelone]
 \StrBetween[1,2]{#2}{,}{,}[\labeltwo]
 \StrBetween[2,3]{#2}{,}{,}[\labelthree]
 \StrBetween[3,4]{#2}{,}{,}[\labelfour]
 \StrBetween[4,5]{#2}{,}{,}[\labelfive]
 \StrBehind[5]{#2}{,}[\labelsix]

 \DDnode{x}{0,-\myb}{};
 \DDnode{w}{1,-\myb}{};
 \DDnode{w}{2,-\myb}{};
 \DDnode{w}{3,-\myb}{};
 \DDnode{w}{4,-\myb}{};
 \DDnode{w}{5,-\myb}{};

 \draw (0,-\myb) node[below=2pt] {\labelone};
 \draw (1,-\myb) node[below=2pt] {\labeltwo};
 \draw (2,-\myb) node[below=2pt] {\labelthree};
 \draw (3,-\myb) node[below=2pt] {\labelfour};
 \draw (4,-\myb) node[below=2pt] {\labelfive};
 \draw (5,-\myb) node[below=2pt] {\labelsix};

 \end{tikzpicture}
 }

\newtheorem{theorem}{Theorem}

\newtheorem{theoremb}{Theorem}
\newtheorem{theoremc}{Theorem}
\newtheorem{theoremd}{Theorem}
\newtheorem{theoreme}{Theorem}
\newtheorem{theoremf}{Theorem}
\newtheorem{dfn}[theoremc]{Definition\!\!}
\newtheorem{rk}[theoremf]{Remark}
\newtheorem{cor}[theoreme]{Corollary}
\newtheorem{lem}[theoremd]{Lemma}
\newtheorem{examp}[theoremc]{Example\!\!}
\newtheorem{prop}[theoremb]{Proposition}
\newenvironment{Proof}[1]{\textbf{#1.} }

\newcommand\bib[1]{\bibitem[#1]{#1}}

\renewcommand\a{\alpha}
\renewcommand\b{\beta}

\newcommand\com[1]{}
\newcommand\C{{\mathbb C}}
\newcommand\Cc{{\let\mathcal\mathscr\mathcal C}}

\newcommand\fa{\mathfrak{a}}
 \newcommand\fg{\mathfrak{g}}
 
 \newcommand\ff{\mathfrak{f}}
 \newcommand\fp{\mathfrak{p}}
 \newcommand\fS{\mathfrak{S}}
 \newcommand\fU{\mathfrak{U}}
\newcommand\g{{\frak g}}

\renewcommand\l{\lambda}

\newcommand\La{\Lambda}

\newcommand\oo{\omega}
\newcommand\op[1]{\mathop{\rm #1}\nolimits}
\newcommand\ot{\otimes}
\newcommand\p{\partial}

\newcommand\R{{\mathbb R}}
\renewcommand\t{\tau}

\newcommand\ti{\tilde}

\newcommand\vp{\varphi}
\newcommand\we{\wedge}
\newcommand\x{\xi}
\newcommand\z{\sigma}
\newcommand\Z{{\mathbb Z}}

\begin{document}

 \title[Submaximally symmetric c-projective structures]{Submaximally symmetric \\ c-projective structures}
 \author{Boris Kruglikov, Vladimir Matveev, Dennis The}
 \date{}
 \address{BK: \ Institute of Mathematics and Statistics, University of Troms\o, Troms\o\ 90-37, Norway.
\quad E-mail: {\tt boris.kruglikov@uit.no}. \newline
 \hphantom{W} VM: \ Institut f\"ur Mathematik, Friedrich-Schiller-Universit\"at, 07737, Jena, Germany.
\quad Email: {\tt vladimir.matveev@uni-jena.de}\newline
 \hphantom{W} DT: \ Mathematical Sciences Institute, Australian National University, ACT 0200, Australia.
\quad E-mail: {\tt dennis.the@anu.edu.au}\newline
 \hphantom{W} DT: \  Fakult\"at f\"ur Mathematik, Universit\"at Wien, Oskar-Morgen\-stern-Platz 1, 1090 Wien, Austria.
\quad E-mail: {\tt dennis.the@univie.ac.at}}
\keywords{Almost complex structure, complex minimal connection, c-projective structure,
submaximal symmetry dimension, pseudo-K\"ahler metric.}
\subjclass[2010]{32Q60, 53C55, 53A20, 53B15, 58J70}

 \begin{abstract}
C-projective structures are analogues of projective\linebreak
structures in the almost complex setting.
The maximal dimension of the Lie algebra of c-projective symmetries of a complex connection on an almost complex
manifold of $\C$-dimension $n>1$ is classically known to be $2n^2+4n$.
We prove that the submaximal dimension is equal to $2n^2-2n+4+2\,\delta_{3,n}$.
If the complex connection is minimal (encoded as a normal parabolic geometry),
the harmonic curvature of the c-projective structure has three components and we specify
the submaximal symmetry dimensions and the corresponding geometric models for each of these three pure curvature types.
If the connection is non-minimal, 
we introduce a modified normalization condition on the parabolic geometry and use this to resolve the symmetry gap problem.
We prove that the submaximal symmetry dimension in the class of Levi-Civita connections for pseudo-K\"ahler metrics is $2n^2-2n+4$,
and specializing to the K\"ahler case, we obtain $2n^2-2n+3$. This resolves the symmetry gap problem for metrizable c-projective structures.
 \end{abstract}

 \maketitle

\section*{Introduction and Main Results}

Let $\nabla$ be a linear connection on a smooth connected almost complex manifold $(M^{2n},J)$ of $\C$-dimension $n\geq2$.
We will assume throughout  $\nabla$ is a {\it complex connection\/}, which means $\nabla J=0$ $\Leftrightarrow$ $\nabla_X(JY)=J\nabla_XY$ for 
all vector fields $X,Y\in\mathcal{D}(M)$. Every almost complex manifold has a complex connection, and the map
$\nabla\mapsto\frac12(\nabla-J\nabla J)$ is a projection from the space of all connections
to the space of complex connections.

The torsion $T_\nabla\in\Omega^2(M)\ot\mathcal{D}(M)$ of a complex connection $\nabla$ needs not vanish.
Its total complex-antilinear part $T_\nabla^{--}\in\Omega^{0,2}(M)\ot\mathcal{D}(M)$ is
equal to $\frac14N_J$, where
 $$
N_J(X,Y)=[JX,JY]-J[JX,Y]-J[X,JY]-[X,Y]
 $$
is the Nijenhuis tensor of $J$.
In particular, for non-integrable $J$, the complex connection $\nabla$ is never symmetric.
However the other parts of $T_\nabla$ can be set to zero by a choice of complex connection.
There always exist {\it minimal\/} connections $\nabla$ characterized by
$T_\nabla=T_\nabla^{--}$, see \cite{Lic}.

Recall that two (real) connections are projectively equivalent if their (unparametrized) geodesics
$\gamma$, given by $\nabla_{\dot\gamma}\dot\gamma\in\langle\dot\gamma\rangle$, are the same
(here $\langle Y\rangle$ denotes the linear span of $Y$ over $C^\infty(M)$).
Thus equivalence $\nabla\sim\bar\nabla$ means $\nabla_XX-\bar\nabla_XX\in\langle X\rangle$
$\forall X\in\mathcal{D}(M)$, and any connection is (real) projectively equivalent to a symmetric one:
$\nabla\simeq\nabla-\frac12T_\nabla$.

A natural and actively studied analogue of projective equivalence in the presence of a complex structure is c-projective equivalence
(also known as h-projective or holomorph-projective equivalence  \cite{Ta,Y,Mi}). Let us recall the basic definitions.
A $J$-planar curve $\gamma$ is given by the differential equation $\nabla_{\dot\gamma}\dot\gamma\in\langle\dot\gamma\rangle_\C
=\langle\dot\gamma,J\dot\gamma\rangle$. Reparame\-tri\-za\-tion does not change this property. Actually, by a reparametrization one can
achieve $\alpha=0$ in the decomposition
 $$
\nabla_{\dot\gamma}\dot\gamma=\alpha\dot\gamma+\beta J\dot\gamma,
 $$
and then $\beta$ is invariant up to a constant multiple (in general, the function $I=(\alpha+\nabla_{\dot\gamma})(\beta^{-1})$ is an
invariant of reparametrizations). Geodesics correspond to $\beta=0$ (singular value for $I$).
As complex analogues of geodesics, $J$-planar curves are of considerable interest in complex
and K\"ahler geometry, see e.g. \cite{Is,ACG,KiT,MR$_1$}.

Two pairs $(J,\nabla)$ on the same manifold $M$ (with $\nabla J=0$)
are called {\it c-projectively equivalent\/} if they
share the same class of $J$-planar curves. It is easy to show that the almost complex structure $J$ is restored up to sign by the c-projective equivalence,
and we will fix the structure $J$ (this does not influence the symmetry algebra)\footnote{By this reason
we sometime talk of c-projective equivalence of complex connections
$\nabla$ on the fixed almost complex background $(M,J)$.}.
Thus we arrive to:

 \begin{dfn}
Two complex connections on an almost complex manifold $(M,J)$ are c-projectively equivalent $\nabla\sim\bar\nabla$
if they have the same $J$-planar curves, i.e.\
$\nabla_XX-\bar\nabla_XX\in\langle X\rangle_\C$ $\forall X\in\mathcal{D}(M)$. 
A {\it c-projective structure\/} is an equivalence class $(M,J,[\nabla])$.
 \end{dfn}

This equivalence can be reformulated tensorially for (complex) connections $\nabla,\bar\nabla$ on $(M,J)$ with equal torsion 
$T_\nabla=T_{\bar\nabla}$ as follows:
$\nabla\sim\bar\nabla=\nabla+\op{Id}\odot\Psi-J\odot J^*\Psi$ for some 1-form $\Psi\in\Omega^1(M)$.
In other words, $\bar\nabla$ is c-projectively equivalent to $\nabla$ if and only if
 $$
\bar\nabla_XY=\nabla_XY+\Psi(X)Y+\Psi(Y)X-\Psi(JX)JY-\Psi(JY)JX
 $$
(notice that $T_\nabla=T_{\bar\nabla}$), see \cite{OT,Is,MS}.

We will show that it is possible to canonically fix the torsion within one c-projective class. 
For minimal connections the torsion is already canonical (=$\frac14N_J$), but in general 
a complex connection is not c-projec\-tive\-ly equivalent to a minimal one 
(in particular, to a symmetric complex connection, cf. the real case). We will demonstrate that the obstruction 
to finding a minimal connection in the c-projective class $[\nabla]$ is the following part of the torsion
 \begin{align*}
T_\text{traceless}^{-+}(X,Y)=&\,\tfrac14\bigl(T_\nabla(X,Y)+JT_\nabla(JX,Y)-JT_\nabla(X,JY)\\
&+T_\nabla(JX,JY)\bigr)-\tfrac1{2n}\bigl(\varsigma(X)Y+\varsigma(JX)JY\bigr),
 \end{align*}
where $\varsigma(X)=\frac12\op{Tr}\bigl(T_\nabla(X,\cdot)+JT_\nabla(JX,\cdot)\bigr)=\tfrac12(X^aT^b_{ab}+J^a_kX^kT^c_{ab}J^b_c)$.
This invariant of c-projective structure is called $\kappa_\text{IV}$ in Section \ref{S4},
where we elaborate the general case (non-minimal connections), and prove an equivalence of categories bet\-ween 
c-projective structures and parabolic geometries of type $\op{SL}(n+1,\C)_\R/P$ with a modified normalization. 

A vector field $v$ is called a {\it c-projective symmetry}
if its local flow $\Phi^v_t$ preserves the class of $J$-planar curves. Equivalently, a c-projective symmetry
is a $J$-holomorphic vector field $v$ such that its local flow transforms $\nabla$ to a c-projectively
equivalent connection: $(\Phi^v_t)^*J=J$, $[(\Phi^v_t)^*\nabla]=[\nabla]$.
The first equation can be re-written as $L_vJ=0$. The second equation, written as $L_v[\nabla]=0$,
can be expressed in local coordinates, with the connection $\nabla$ given by the Christoffel symbols\footnote{Our index convention is that
$\nabla_{\p_j}\p_k=\Gamma_{jk}^i\p_i$.} 
$\Gamma_{jk}^i$, as follows:
 $$
\Omega^i_{jk}- \phi_j\delta^i_k- \phi_k\delta^i_j+ \phi_\alpha J^\alpha_j J^i_k+ \phi_\alpha J^\alpha_k J^i_j =0,
 $$
where $\Omega_{jk}^i=L_v(\Gamma)^i_{jk}$ and $\phi_j=\frac1{2(n+1)}\Omega^i_{ji}$
(notice that this manifestly non-symmetric formula implies $L_vT_\nabla=0$).
We use these equations in computing symmetries of the explicit models.

The space of c-projective vector fields forms a Lie algebra, denoted $\mathfrak{cp}(\nabla,J)$.
It is well known (and we recall in the next section) that the maximal dimension of this algebra
is equal to $2n^2+4n$, and this bound is achieved only
if the structure is {\it flat\/}, i.e.\ c-projectively locally equivalent to $\C P^n$ equipped
with the standard complex structure $J_\text{can}$ and the class of the Levi-Civita connection
$\nabla^\text{FS}$ of the Fubini-Study metric. Indeed, the group of c-projective symmetries of this
flat structure $(\C P^n,J_\text{can},[\nabla^\text{FS}])$
is $\op{PSL}(n+1,\C)$, and its Lie algebra is $\mathfrak{sl}(n+1,\C)$.

For many geometric structures the natural (and often nontrivial) problem is to compute the
next possible/realizable dimension, the so-called {\it submaximal dimension\/},
of the algebra of symmetries, see \cite{E$_2$,Ko,K$_3$,KT} and the references therein.

For the algebra of (usual) projective vector fields the question was settled in \cite{E$_1$}.
For c-projective vector fields the answer is as follows.

 \begin{theorem}\label{Thm1}
Consider a c-projective structure $(M,J,[\nabla])$. If it is not flat\footnote{That is in a neighborhood of at least one point of $M$ the c-projective
structure is not locally equivalent to $(\C P^n,J_\text{can},[\nabla^\text{FS}])$.}, then $\dim\mathfrak{cp}(\nabla,J)$ is bounded from above by
 $$
\mathfrak{S}=\left\{
\begin{array}{ll}2n^2-2n+4,& n\neq3,\\ 18,& n=3.\end{array}
\right.
 $$
and this estimate is sharp (= realizable).
 \end{theorem}

We will show that the dimensional bound $2n^2-2n+4$ is realizable via both non-minimal and minimal complex connections.

Let us now discuss the minimal case. By \cite{H,CEMN} the corresponding c-projective structures can be encoded as
{\it (regular\footnote{For $|1|$-graded
geometries the regularity condition is vacuous; in particular it can be dropped for c-projective structures.}) normal parabolic geometries\/} of
type $\op{SL}(n+1,\C)_\R/P$, we will recall the setup in the next section. The fundamental invariant of any regular normal parabolic geometry
is its harmonic curvature $\kappa_H$, through which the flatness writes simply as $\kappa_H=0$.
As will be discussed in the next section, for c-projective structures $(J,[\nabla])$ with minimal $\nabla$
the harmonic curvature has three irreducible components $\kappa_H=\kappa_{\rm{I}}+\kappa_{\rm{II}}+\kappa_{\rm{III}}$.

According to \cite{KT}, the submaximal dimension is attained when only one of the components of
the curvature is non-zero (provided the universal upper bound is realized, see loc.cit.\ for the precise statement; in our case this condition is satisfied). Thus we can study a finer question, namely
what is the maximal dimension of the algebra of c-projective
vector fields, in the case the curvature is non-zero and has one of the types I-III.
Let $\mathfrak{S}_i$ be the maximal dimension of the algebra $\mathfrak{cp}(\nabla,J)$
in the case $\nabla$ is not flat, and its curvature has fixed type $i$.

\begin{theorem}\label{Thm2}
For c-projective structures $(M,J,[\nabla])$, associated with minimal complex connections $\nabla$, the submaximal dimension
of $\mathfrak{cp}(\nabla,J)$ within a fixed curvature type is equal to
 $$
\mathfrak{S}_{\rm{II}}=2n^2-2n+4.
 $$
 $$
\mathfrak{S}_{\rm{I}}=\left\{
\begin{array}{ll}2n^2-4n+10,& n>2,\\ 6,& n=2.\end{array}
\right.
\quad
\mathfrak{S}_{\rm{III}}=\left\{
\begin{array}{ll}2n^2-4n+12,& n>2,\\ 8,& n=2.\end{array}
\right.
 $$
 \end{theorem}

Let us list the first values of the submaximal dimensions:

\smallskip
 \begin{center}
\begin{tabular}{c||c|c|c|c|c|c}
SubMax Dim & $n=2$ & $n=3$ & $n=4$ & $n=5$ & $n=6$ & \dots\\
\hline
Type {\rm{I}} & 6 & 16 & 26 & 40 & 58 & \dots \\
\hline
Type II & 8 & 16 & 28 & 44 & 64 & \dots \\
\hline
Type III & 8 & 18 & 28 & 42 & 60 & \dots
\end{tabular}
 \end{center}

Sharpness in the dimension estimates will be obtained by exhibiting the explicit
models and their symmetries, and we get $\mathfrak{S}=\max\mathfrak{S}_i$.

 \begin{cor}\label{Cor}
Consider a complex manifold $(M,J)$ with a complex symmetric connection $\nabla$.
If the c-projective structure $(J,[\nabla])$ is not flat, then its symmetry dimension
does not exceed\/ $\mathfrak{S}_0=2n^2-2n+4$ and this upper bound is realizable.
 \end{cor}

On the way to proving Theorem \ref{Thm2} we establish two general results about
the symmetry gap problem for {\it real\/} parabolic geometries (Propositions
\ref{P:lw-vec} and \ref{P:PR}), which generalize some results of \cite{KT} and are of independent interest.

An important problem in projective differential geometry is to determine if a given projective connection is metrizable.
In the c-projective case, the corresponding problem is to determine if the structure $(J,[\nabla])$ is represented by
the Levi-Civita connection $\nabla^g$ of a pseudo-K\"ahler\footnote{By this we mean (throughout the paper) both indefinite and definite cases.}
structure $(g,J)$, where $g$ is a metric and $J$ a complex structure (related by $J^*g=g$, $\nabla^gJ=0$; in particular, $J$ is integrable).
For such structures we also compute the submaximal symmetry dimension.

 \begin{theorem}\label{Thm3}
For a K\"ahler structure $(M,g,J)$ of non-constant holomorphic sectional curvature $\dim\mathfrak{cp}(\nabla^g,J)\le 2n^2-2n+3$.
This bound is realized by $(M=\C P^1\times\C^{n-1},J=i)$ with its natural K\"ahler metric.

For a pseudo-K\"ahler structure $(M,g,J)$ of non-constant holomorphic sectional curvature we have: $\dim\mathfrak{cp}(\nabla^g,J)\le 2n^2-2n+4$.
This estimate is sharp in any signature $(2p,2(n-p))$, $0<p<n$.
 \end{theorem}

Thus the submaximal symmetry dimension $\mathfrak{S}_0=2n^2-2n+4$ from the above corollary is realizable by a pseudo-K\"ahler metric.
In fact, the submaximal c-projective structure with complex $J$ and symmetric connection $\nabla$ preserving $J$ and having curvature type II is unique
and metrizable (compare this to the real projective case \cite{KM}, where the submaximal structure is not metrizable). 
The corresponding pseudo-K\"ahler metric(s), given by formula (\ref{subMKh}), will be described in detail.

It seems plausible that the above result about K\"ahler structures extends to a larger space of c-projective structures
associated to almost Hermitian pairs $(g,J)$. These are given by $\nabla$ obtained uniquely from the conditions: $\nabla g=0$, $\nabla J=0$.
We conjecture that all submaximal c-projective structures in this class are associated to K\"ahler structures.

In Appendix \ref{S.A} we give a detailed account of how the submaximal model for an exceptional case (type III, $n=2$)
is constructed. We discuss the uniqueness issue of the submaximal models in Appendix \ref{S.B}.

\section{C-projective structures: the minimal case.}\label{S1}

In this section we give the necessary background on c-projective equivalence
of {\it minimal\/} complex connections $\nabla$ on an almost complex manifold $(M,J)$
of $\dim\!M=2\dim_\C\!M=2n$ ($n>1$).

Such c-projective structures on $2n$-dimensional manifolds are the underlying structures of
regular {\it normal\/} parabolic geometries of type $G/P$, where $G=SL(n+1,\C)$ and $P$ is
the subgroup that stabilizes a complex line $\ell\subset\C^{n+1}$;
both $G$ and $P$ are to be regarded as {\it real\/} Lie groups.

We recall some basic setup, referring to \cite{CS,Y,H} for further details.
The parabolic subgroup $P$ induces the Lie algebra gradation
on the space of trace-free complex matrices:
 \[
 \g=\mathfrak{sl}(n+1,\C)_\R = \g_{-1}\oplus\overbrace{\g_0\oplus\g_1}^\mathfrak{p},
 \]
If $\ell$ is spanned by the first standard basis vector in $\C^{n+1}$, then
 \com{
  $$
\g_{-}=\left(\begin{array}{c|ccc}
 0 & 0 & \cdots & 0 \\ \hline
 * & 0 & \cdots & 0\\
 \vdots & \vdots & \ddots & \vdots \\
 * & 0 & \cdots & 0
 \end{array} \right),\quad
 \g_0=\left(\begin{array}{c|ccc}
 * & 0 & \cdots & 0 \\ \hline
 0 & * & \cdots & *\\
 \vdots & \vdots & \ddots & \vdots \\
 0 & * & \cdots & *
 \end{array} \right),\quad
\g_+=(\g_{-})^t.
 $$ }
   \[
 \g_{-}= \left(\begin{array}{c|c}
 0 & 0  \\ \hline
 * & 0 \\
 \end{array} \right), \quad
  \g_0 = \left(\begin{array}{c|c}
 * & 0  \\ \hline
 0 & * \\
 \end{array} \right), \quad
 \g_{+}= \left(\begin{array}{c|c}
 0 & *  \\ \hline
 0 & 0 \\
 \end{array} \right),
 \]
using the blocks of size $1$ and $n$ along the diagonal.  In fact, the gradation is induced by a (unique) grading element $Z \in \mathfrak{z}(\g_0)$, i.e.\ $\g_j$ is the eigenspace with {\em homogeneity} (eigenvalue) $j$ for $\operatorname{ad}_Z$. With the standard choice
 $Z=\op{diag}(\frac{n}{n+1},\frac{-1}{n+1},\dots,\frac{-1}{n+1})$.

The fundamental invariant of any regular normal parabolic geometry is its harmonic
curvature $\kappa_H$, whose vanishing (flatness) is the complete obstruction to
local equivalence to the homogeneous model $G/P$. For c-projective structures $\kappa_H=0$
is equivalent to $(M,J,[\nabla])$ being locally isomorphic to $(\C P^n,J_\text{can},[\nabla^\text{FS}])$.

In this (flat) case only, dimension of the symmetry algebra $\mathfrak{cp}(\nabla,J)$ is equal
$\dim_\R\g=2(n^2+2n)$. Otherwise the dimension is strictly smaller and we obtain the gap of dimensions: $\dim\g-\mathfrak{S}$.

The harmonic curvature $\kappa_H$ takes values in the space $\mathbb{V}=H^2_+(\g_{-},\g)$
consisting of all positive homogeneity components of the Lie algebra cohomology $H^2(\g_-,\g)$
with respect to the natural $\g_0$-action. For c-projective structures, $\mathbb{V}$ decomposes
as a $\g_0$-module into irreducibles (irreps):
$\mathbb{V}=\mathbb{V}_{{\rm I}}\oplus\mathbb{V}_{{\rm II}}\oplus\mathbb{V}_{{\rm III}}$
(here subscripts are mere numerations). Using the standard $(p,q)$-notation for the decomposition
of tensors with respect to the almost complex structure $J$, we have\footnote{Below and throughout $A\odot B$ denotes 
the Cartan product of $A$ and $B$, i.e.\ the highest weight irreducible submodule of $A\ot B$ ("traces removed"). If both modules
$A,B$ are complex then we can also form tensor/Cartan product over $\C$.} \cite{CEMN}:
 $$
\mathbb{V}_{{\rm I}}=\left\{\begin{array}{ll}
\La^{2,0}\g_{-}^*\odot_{\C}\mathfrak{sl}(\g_{-},\C), & n>2,\\
\La^{2,0}\g_{-}^*\odot_{\C}\g_{-}^*, & n=2;\end{array}\right.
 $$
 $$
\mathbb{V}_{{\rm II}}=\La^{1,1}\g_{-}^*\odot\mathfrak{sl}(\g_{-},\C);
\qquad \mathbb{V}_{{\rm III}}=\La^{0,2}\g_{-}^*\ot_{\C}\g_{-}\simeq\La^2\g_{-}^*\ot_{\bar\C}\g_{-}.
 $$
In the standard terminology, 
$\mathbb{V}_{{\rm I}}\oplus\mathbb{V}_{{\rm II}}$ is the space of curvatures, and
$\mathbb{V}_{{\rm III}}$ is the space of torsions.
With respect to $\g_0$-action, $\mathbb{V}_{{\rm I}}$ has homogeneity $2+\delta_{2,n}$,
$\mathbb{V}_{{\rm II}}$ has homogeneity $2$, and $\mathbb{V}_{{\rm III}}$ has homogeneity $1$.

The harmonic curvature splits in accordance to the above into irreducible components
(projections to which are the usual symmetrizers)
 $$
\kappa_H=\kappa_{\rm{I}}+\kappa_{\rm{II}}+\kappa_{\rm{III}},
 $$
where (we refer to \cite{CEMN} for explicit formulae;
we only need to know the tensorial type to prove Theorem \ref{Thm2})
 \begin{itemize}
 \item $\kappa_{\rm{I}}$ is the $(2,0)$-part of Weyl projective curvature of $\nabla$ for $n > 2$,
or the $(2,0)$-part of the Liouville tensor when $n=2$;
 \item $\kappa_{\rm{II}}$ is the $(1,1)$-part of Weyl projective curvature tensor of $\nabla$;
 \item $\kappa_{\rm{III}}$ is $\frac14N_J$ (torsion of a minimal complex connection $\nabla$).
 \end{itemize}
We remark that on a complex background $(M,J)$ ($\kappa_{\rm{III}} = 0$):
 \begin{itemize}
 \item Existence of a holomorphic connection in $[\nabla]$ is equivalent to $\kappa_{\rm{II}} = 0$ (see \cite[Prop.\ 3.1.17]{CS}).
 \item $\kappa_{\rm{I}} = 0$ is a necessary condition for $(M,J,[\nabla])$ to be (pseudo-) K\"ahler
 metrizable (the curvature is of type (1,1)). 
 \end{itemize}

We now summarize an abstract description of
$\mathbb{V}_{{\rm I}}, \mathbb{V}_{{\rm II}}, \mathbb{V}_{{\rm III}}$ that will be used in the sequel.
The Satake diagram encoding the real Lie algebra $\g = \mathfrak{sl}(n+1,\C)_\R$ has $n$ nodes in the top and bottom rows:
 \[
\SLCR[].
 \]
Dynkin diagram of the complexification
$\g_\C \cong \mathfrak{sl}(n+1,\C) \times \mathfrak{sl}(n+1,\C)$ is obtained by removing
all arrows from the above Satake diagram.
As $\g$-modules, $\g_\C \cong \g \oplus \bar\g$, where $\g$ and $\bar\g$ correspond
to the $\pm i$-eigenspaces for the natural $\g$-invariant complex structure on $\g$.
Pictorially, $\g$ and $\bar\g$ correspond respectively to the top and bottom rows of the Satake diagram,
and conjugation swaps these factors by reflection in the indicated arrows.
The original real Lie algebra $\g$ is naturally identified with the fixed point set under conjugation,
i.e.\ $\g \cong \{ x + \overline{x} : x \in \g \}$.

The choice of parabolic $\mathfrak{p}\subset\mathfrak{sl}(n+1,\C)_\R$ is encoded by marking the Satake diagram with crosses:
 $$
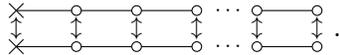
\CPgen[].
 $$
 The Satake diagram of the semisimple part of $\g_0$ is obtained by removing the crossed nodes:
$(\g_0)_\text{ss}\simeq\mathfrak{sl}(n,\C)_\R$. The parabolic $\mathfrak{p}_\C \subset \g_\C$
induces a grading of $\g_\C$ and we have
$\mathbb{V}_\C = H^2(\g_-,\g) \otimes \C \cong H^2((\g_\C)_-, \g_\C)$.
Using Kostant's version of the Bott--Borel--Weil theorem \cite{BE,CS}, the computation of $(\g_\C)_0$-module structure of $H^2((\g_\C)_-, \g_\C)$ is algorithmically straightforward.  Namely, each $(\g_\C)_0$-irrep, denoted $\mathbb{W}_\mu$, occurs with multiplicity one and its {\em lowest} weight is $\mu=-w\cdot\nu$, where\footnote{When working with the complexification $\g_\C$, we use barred quantities in association with the second (bottom) $\mathfrak{sl}(n+1,\C)$ factor.}
 \begin{itemize}
 \item we use the affine action $\cdot$ of the Weyl group of $\g_\C$ on $\g_\C$-weights.
 \item $w = (jk)$ is a length two element of the Hasse diagram \cite{BE,CS} of $(\g_\C,\mathfrak{p}_\C)$. Here: $w=(12)$, $(1\bar{1})$, or $(\bar{1}\bar{2})$.
 \item $\nu$ is the highest (minus lowest) weight of (the adjoint representation of) a simple ideal in $\g_\C$.  Here: The highest weight of $\mathfrak{sl}(n+1,\C)$ is $\lambda = \lambda_1 + \lambda_n$, expressed in terms of the fundamental weights $\{ \lambda_i \}$, and we have $\nu = \lambda$ or $\nu = \bar\lambda$.
  \end{itemize}
 We encode $\mu$ as follows: express $-\mu$ in terms of the fundamental weights of $\g_\C$ and mark a given node of the Dynkin diagram of $\g_\C$ with its corresponding coefficient \cite{BE}.  Here, $\mathbb{V}_\C$ decomposes into six $(\g_\C)_0$-irreps occurring in three conjugate pairs, and this accounts for the three $\g_0$-irreps in $\mathbb{V}$.  For the real case, we take the same marked Dynkin diagrams but now include the arrows so as to obtain a marked Satake diagram.  Conjugate copies are indicated by the symbol $\operatorname{Cc}$.

\begin{table}[h]
 $$
\hskip-7pt\begin{array}{c|c|c|c} \hline
 \mbox{\small{Type}} & n > 3 & n =3 & n =2 \\ \hline
 \mbox{I} & \CPgen{-4,1,1,0,0,1}{0,0,0,0,0,0} \oplus \mbox{Cc}  & \CPthree{-4,1,2}{0,0,0}\oplus \mbox{Cc} & \CPtwo{-5,1}{0,0}\oplus \mbox{Cc} \\
 \mbox{II} & \CPgen{-3,2,0,0,0,1}{-2,1,0,0,0,0} \oplus \mbox{Cc} & \CPthree{-3,2,1}{-2,1,0}\oplus \mbox{Cc} & \CPtwo{-3,3}{-2,1}\oplus \mbox{Cc} \\
 \mbox{III} & \CPgen{1,0,0,0,0,1}{-3,0,1,0,0,0} \oplus \mbox{Cc} & \CPthree{1,0,1}{-3,0,1}\oplus \mbox{Cc} & \CPtwo{1,1}{-3,0}\oplus \mbox{Cc} \\ \hline
\end{array}
 $$
 \caption{Irreducible harmonic curvature components}
 \label{F:Kh-comp}
\end{table}

Each $\g_0$-irrep, denoted as $\mathbb{V}_\mu$, complexifies to $(\mathbb{V}_\mu)_\C \cong \mathbb{W}_\mu \oplus \overline{\mathbb{W}_\mu}$ as $(\g_\C)_0$-irrep (and $\g_0$-irrep).  Here, $\mathbb{V}_\mu$ is identified with its fixed point set under conjugation, i.e.\ $\mathbb{V}_\mu \cong \{ \phi + \overline{\phi} : \phi \in \mathbb{W}_\mu \}$.  Kostant's theorem explicitly describes a lowest weight vector $\phi_0$ in each $(\g_\C)_0$-irrep $\mathbb{W}_\mu$.  Without loss of generality, $\mu = -w \cdot \lambda$ with $w = (jk)$.  Then
 \[
 \phi_0=e_{\alpha_j}\we e_{\z_j(\alpha_k)}\ot v,
 \]
 in terms of root vectors $e_\beta$, simple roots $\{ \alpha_j \}$, the simple reflection $\sigma_j$, and $v \in \fg_\C$ a weight vector having weight $-w(\lambda)$.

 \begin{table}[h]
  \[
\begin{array}{c|c|c}
 \mbox{{\small Type}} & w & \phi_0 \\ \hline
\mbox{I} & (12) & \left\{ \begin{array}{ll} e_{\alpha_1} \wedge e_{\alpha_1+\alpha_2} \otimes e_{-\alpha_2-\dots-\alpha_n}, & n > 2; \\ e_{\alpha_1} \wedge e_{\alpha_1+\alpha_2} \otimes e_{\alpha_1}, & n = 2 \end{array} \right. \\
\mbox{II} & (1\bar{1}) &  e_{\alpha_1} \wedge e_{\overline{\alpha}_1} \ot  e_{-\alpha_2-\dots-\alpha_n} \\
\mbox{III} & (\bar{1}\bar{2}) & e_{\overline{\alpha}_1} \we e_{\overline{\alpha}_1+\overline{\alpha}_2} \ot e_{-\alpha_1-\dots-\alpha_n} \\ \hline
\end{array}
 \]
 \caption{Lowest weight vectors for harmonic curvature modules}
 \label{F:Kh-lw}
 \end{table}

\section{Upper bound on the submaximal symmetry dimension}\label{S2}

A universal upper bound $\mathfrak{U}$ on the submaximal symmetry dimension $\mathfrak{S}$ for
regular normal parabolic geometries of type $(G,P)$ was proved in \cite{KT}.  In terms of $\mathbb{V} = H^2_+(\g_-,\g)$, we have
 $$
 \mathfrak{U} := \max\{\dim(\fa^\psi) :\,0\neq\psi\in \mathbb{V}\},
 $$
where $\fa^\psi$ is the {\em Tanaka prolongation}
of the pair $(\g_{-},\fa_0=\mathfrak{ann}_{\g_0}(\psi))$ in $\g$.
Namely, $\fa^\psi = \g_- \oplus \fa_0 \oplus \fa_+$ is the graded Lie subalgebra of $\g$ with
 \begin{equation}\label{E:a-phi}
\fa_k = \{ X \in \g_k : \operatorname{ad}^k_{\g_{-1}}(X) \,\cdot \,\psi = 0 \},\ k\ge1.
 \end{equation}

To calculate $\mathfrak{U}$ it suffices to decompose $\mathbb{V}$ into $\g_0$-irreps, calculate the
corresponding maximum for each submodule, and then take the maximum of these.  The calculation becomes particularly easy for those $(G,P)$
that are {\em prolongation-rigid\/} (as defined in \cite{KT}):
for any $0 \neq \psi \in \mathbb{V} = H^2_+(\g_-,\g)$, we have $\fa^\psi_+ = 0$, so that
$ \fa^\psi = \g_- \oplus \mathfrak{ann}_{\g_0}(\psi)$.

In \cite{KT}, the {\em complex\/} case was thoroughly investigated.  In particular, if $(\g,\fp)$ are complex Lie algebras and $\mathbb{V}_\mu \subset \mathbb{V}$ is a $\g_0$-irrep with lowest weight vector $\phi_0$ (and lowest weight $\mu$), then it was proved in \cite{KT} that
 \begin{enumerate}
 \item[(i)] $\fU_\mu = \max\{\dim(\fa^\psi) :\,0\neq\psi\in \mathbb{V}_\mu \}$ is realized by $\dim(\fa^{\phi_0})$;
 \item[(ii)] $\fa^{\phi_0}_+ = 0$ if and only if all integers above crossed nodes for $\mu$ are {\it nonzero\/}.
In this case, $\fU_\mu = \dim(\g_-) + \dim(\mathfrak{ann}(\phi_0))$.
 \end{enumerate}
 If (ii) is satisfied for each $\mathbb{V}_\mu \subset \mathbb{V}$, then $(G,P)$ is prolongation-rigid.

We now consider the case of general {\em real\/} Lie groups underlying given complex Lie groups $(G,P)$,
and refer to the marked Satake diagram notation as before (see \cite{CS}).
The {\em complexification} of any given real $G_0$-irrep
$\mathbb{V}_\mu\subset\mathbb{V} = H^2_+(\g_-,\g)$ is either:
 \begin{enumerate}
 \item[(i)] $\mathbb{W}_\mu \cong \overline{\mathbb{W}_\mu}$, or
 \item[(ii)] $\mathbb{W}_\mu \oplus \overline{\mathbb{W}_\mu}$ (if $\mathbb{W}_\mu \not\cong \overline{\mathbb{W}_\mu}$)
 \end{enumerate}
for some $\g_\C$-weight $\mu$.
In either case, we will (abuse notation and) refer to the given (real) $G_0$-irrep as $\mathbb{V}_\mu$.
Note that (i) occurs if and only $\mu$ is self-conjugate.  For c-projective structures, only (ii) occurs.
Defining $\fU_\mu = \max\{\dim(\fa^\psi) :\,0\neq\psi\in \mathbb{V}_\mu \}$,
where now $\fa^\psi$ is a {\em real} Lie algebra, we respectively have:
 \begin{enumerate}
 \item[(i)] $\fU_\mu = \max\{\dim(\fa^{\phi}) :\,0\neq\phi\in \mathbb{W}_\mu\}$;
 \item[(ii)] $\fU_\mu = \max\{\dim(\fa^{\phi + \bar\phi}) :\,0\neq\phi\in \mathbb{W}_\mu\}$.
 \end{enumerate}

 The following general result is based on \cite[Prop.\ 3.1.1]{KT}.

 \begin{prop}\label{P:lw-vec} 
Let $G$ be a complex semisimple Lie group, and let $P$ be a parabolic subgroup with reductive part $G_0$.
Let $\mathbb{W}$ be a (complex) $G_0$-irrep with $\phi_0 \in \mathbb{W}$ an extremal weight vector.  Regarding $G$ and $P$ as real Lie groups,
we have for $k \geq 0$ and any $0 \neq \phi \in \mathbb{W}$:
 \begin{enumerate}
 \item[(i)] if $\mathbb{W} \cong \overline{\mathbb{W}}$: $\dim(\fa_k^{\phi}) \leq \dim(\fa_k^{\phi_0})$;
 \item[(ii)] if $\mathbb{W} \not\cong \overline{\mathbb{W}}$: $\dim(\fa_k^{\phi + \bar\phi}) \leq \dim(\fa_k^{\phi_0 + \overline{\phi_0}})$.
 \end{enumerate}
 \end{prop}

 \begin{proof}  We prove (ii).
 Fix $k \geq 0$, and let $\psi = \phi + \bar\phi$.  From \eqref{E:a-phi}, $\fa_k^\psi = \ker(M(\psi))$, where $M(\psi)$ is some real matrix that depends $\R$-linearly on $\psi$.  The rank of a matrix is a lower semi-continuous function of its entries, so the function $\mathcal{F} : \mathbb{W} \to \mathbb{Z}$ given by
 $\mathcal{F}(\phi) = \dim_\R(\fa_k^{\phi + \bar\phi})$ is upper semi-continuous.  Clearly, $\mathcal{F}(c \phi) = \mathcal{F}(\phi)$ for any $c \in \R^\times$.  Note that  $\mathcal{F}$ is constant on $G_0$-orbits, and since $\mathfrak{z}(\g_0)$ contains a grading element, then $G_0$ contains elements that act on $\mathbb{W}$ by arbitrary $c \in \C^\times$.  Thus, $\mathcal{F}$ descends to the {\em complex} projectivization $\mathbb{P}(\mathbb{W})$.

 It is well-known that $\mathbb{P}(\mathbb{W})$ contains a {\em unique} closed $G_0$-orbit, namely $\mathcal{O} = G_0 \cdot [\phi_0]$.  Thus, $\mathcal{O}$ is in the closure of {\em every} $G_0$-orbit in $\mathbb{P}(\mathbb{W})$.  Hence, since $\mathcal{F} : \mathbb{P}(\mathbb{W}) \to \mathbb{Z}$ is upper semi-continuous and constant on $G_0$-orbits, then (ii) follows immediately.  Proving (i) is similar.
 \end{proof}

 \begin{prop}\label{P:PR}
Let $G$ be a complex semisimple Lie group, and let $P$ be a parabolic subgroup with reductive part $G_0$.  Regard $G$ and $P$ as real Lie groups.  Then $(G,P)$ is prolongation-rigid if 
for every $G_0$-irrep $\mathbb{V}_\mu \subset \mathbb{V} = H^2_+(\g_-,\g)$
the integers over every pair of crossed nodes on the Satake diagram of $\mu$
joined by an arrow are not both zero.
 \end{prop}

 \begin{proof} It suffices to prove the result for a single $\mathbb{V}_\mu$.  We have the $\g_0$-reps $(\mathbb{V}_\mu)_\C \cong \mathbb{W}_\mu$ or $(\mathbb{V}_\mu)_\C \cong \mathbb{W}_\mu \oplus \overline{\mathbb{W}_\mu}$.  Regard these as $(\g_\C)_0$-reps.  Repeating the proof of Proposition \ref{P:lw-vec}, but now in the complex case, we find that the complex Lie algebra $\fa^\psi_k$ for each $k \geq 0$ has maximum dimension among $0 \neq \psi \in \mathbb{V}_\mu$ when $\psi = \phi_0$ or $\psi = \phi_0 + \overline{\phi_0}$ respectively.  If the Satake (hence the Dynkin) diagram for $\mu$ satisfies the given condition, then by \cite[eqn (3.2) and Thm.\ 3.3.3]{KT}, we have $\fa^{\phi_0}_+ = 0$ or  $\fa^{\phi_0 + \overline{\phi_0}}_+ = 0$
respectively.  This being true for each $\mathbb{V}_\mu$ forces prolongation-rigidity of $(\g_\C,\mathfrak{p}_\C)$, and hence prolongation-rigidity of $(\g,\mathfrak{p})$.
 \end{proof}

From Table \ref{F:Kh-comp}, we immediately see that the criteria of Proposition \ref{P:PR}
are satisfied for (minimal) c-projective structures.

 \begin{cor}\label{cor2}
C-projective geometry is prolongation-rigid. \qed
 \end{cor}

We will use the notations $\fS_i$ and $\fU_i$ referring to a specific curvature type.
Thus, for each c-projective type, since $2n = \dim(\g_-)$, we have
 \begin{align} \label{E:SA}
 \mathfrak{S}_\mu \leq \mathfrak{U}_\mu = \dim(\fa^{\phi_0 + \overline{\phi_0}}) = 2n + \dim(\mathfrak{ann}_{\g_0}(\phi_0 + \overline{\phi_0})).
 \end{align}
 Using the data from Tables \ref{F:Kh-comp} and \ref{F:Kh-lw}, the annihilators $\fa_0 = \mathfrak{ann}_{\g_0}(\phi_0 + \overline{\phi_0})$ for all three types are computed in Table \ref{F:ann}.

 \begin{footnotesize}
 \begin{center}
 \begin{table}[h]
 \[
 \begin{array}{|c|c|c|} \hline
 \mbox{Type} & \fa_0 \,\, (n \geq 3) & \fa_0 \,\, (n=2) \\ \hline
 \mbox{I} & \begin{array}{c}
 \left(\begin{array}{c|cccccc}
 a_0 & 0 & 0 & 0 & \cdots & 0 & 0\\ \hline
 0 & a_1 & 0 & 0 & \cdots & 0 & 0\\
 0 & * & a_2 & 0 & \cdots & 0 & 0\\
 0 & * & * & a_3 & \cdots & * & 0\\
 \vdots &  \vdots & \vdots & \vdots & \ddots & \vdots & 0\\
 0 & * & * & * & \cdots & a_{n-1} & 0\\
 0 & * & * & * & \cdots & * & a_n\\
 \end{array} \right) \\
 a_0+\dots+a_n=0,\\
 2(a_0-a_1)-a_2+a_n=0,\\
 \dim_\R(\fa_0) = 2(n^2-3n+5)
 \end{array} &
 \begin{array}{c}
 \left(\begin{array}{c|cc}
 a_0 & 0 & 0 \\ \hline
 0 & a_1 & 0 \\
 0 & * & a_2 \\
 \end{array} \right) \\
 a_0 + a_1 + a_2 = 0,\\
 3a_0 - 2a_1 - a_2 = 0,\\
 \dim_\R(\fa_0) = 4
 \end{array}\\ \hline
 \mbox{II} &
 \begin{array}{c}
 \left(\begin{array}{c|cccccc}
 a_0 & 0 & 0 & 0 & \cdots & 0 & 0\\ \hline
 0 & a_1 & 0 & 0 & \cdots & 0 & 0\\
 0 & * & a_2 & * & \cdots & * & 0\\
 0 & * & * & a_3 & \cdots & * & 0\\
 \vdots &  \vdots & \vdots & \vdots & \ddots & \vdots & 0\\
 0 & * & * & * & \cdots & a_{n-1} & 0\\
 0 & * & * & * & \cdots & * & a_n\\
\end{array} \right)\\
 a_0+\dots+a_n=0,\\
 2\op{Re}(a_0-a_1)=a_1-a_n,\\
 \dim_\R(\fa_0) = 2(n^2-2n+2)
 \end{array} &
 \begin{array}{c}
 \left(\begin{array}{c|cc}
 a_0 & 0 & 0 \\ \hline
 0 & a_1 & 0 \\
 0 & * & a_2 \\
 \end{array} \right) \\
 a_0 + a_1 + a_2 = 0,\\
 2\op{Re}(a_0-a_1) = a_1 - a_2,\\
 \dim_\R(\fa_0) = 4
 \end{array}\\ \hline
 \mbox{III} &
 \begin{array}{c}
 \left(\begin{array}{c|cccccc}
 a_0 & 0 & 0 & 0 & \cdots & 0 & 0\\ \hline
 0 & a_1 & * & 0 & \cdots & 0 & 0\\
 0 & * & a_2 & 0 & \cdots & 0 & 0\\
 0 & * & * & a_3 & \cdots & * & 0\\
 \vdots &  \vdots & \vdots & \vdots & \ddots & \vdots & 0\\
 0 & * & * & * & \cdots & a_{n-1} & 0\\
 0 & * & * & * & \cdots & * & a_n\\
\end{array} \right)\\
 a_0+\dots+a_n=0,\\
 2\overline{a_0}-a_0=\overline{a_1}+\overline{a_2}-a_n\\
 \dim_\R(\fa_0) = 2(n^2-3n+6)
 \end{array} &
 \begin{array}{c}
 \left(\begin{array}{c|cc}
 a_0 & 0 & 0 \\ \hline
 0 & a_1 & 0 \\
 0 & * & a_2 \\
 \end{array} \right) \\
 a_0 + a_1 + a_2 = 0,\\
 2\overline{a_0}-a_0=\overline{a_1}+\overline{a_2}-a_2,\\
 \dim_\R(\fa_0) = 4
 \end{array}\\ \hline
 \end{array}
 \]
 \caption{Annihilators $\fa_0 = \mathfrak{ann}_{\g_0}(\phi_0 + \overline{\phi_0})$ associated to harmonic curvature types}
 \label{F:ann}
 \end{table}
 \end{center}
 \end{footnotesize}

 The following recipe for computing $\mathfrak{ann}_{\g_0}(\phi_0 + \overline{\phi_0})$ is analogous to those discussed in \cite{KT}:
 \begin{enumerate}
 \item Put asterisks over any pair of uncrossed nodes connected by an arrow in the Satake diagram
of $\mu$, if one of the nodes has a nonzero coefficient.
This determines an (opposite\footnote{Standard parabolics in this paper are (block) upper triangular, but since we use {\em lowest} weight vectors, the annihilators in $\g_0$ have (block) lower triangular shape.}) parabolic in $(\g_0)_{ss}$ and hence the general shape of $\mathfrak{ann}_{\g_0}(\phi_0 + \overline{\phi_0})$.
 \item  Diagonal elements $X \in \mathfrak{sl}(n+1,\C)_\R$ satisfy $\mu(X) = 0$, since
  \[
 X \cdot (\phi_0 + \overline{\phi_0}) = \mu(X) \phi_0 + \overline{\mu}(X) \overline{\phi_0}.
 \]
 This condition becomes clear by converting $\mu$ into root notation.
 \end{enumerate}

 \begin{examp}
 Consider the type III case when $n \geq 3$.  Then
 \[
 \CPgen{1,0,0,0,0,1}{-3,0,1,0,0,0} \quad\leadsto\quad
 \begin{tikzpicture}[scale=\myscale,baseline=-3pt]

 \tikzstyle{every node}=[font=\tiny]
 \newcommand\myb{0.3}
 \bond{0,\myb};
 \bond{1,\myb};
 \bond{2,\myb};
 \tdots{3,\myb};
 \bond{4,\myb};

 \draw (0,-\myb) node[above=-1pt] {$\updownarrow$};
 \draw (1,-\myb) node[above=-1pt] {$\updownarrow$};
 \draw (2,-\myb) node[above=-1pt] {$\updownarrow$};
 \draw (3,-\myb) node[above=-1pt] {$\updownarrow$};
 \draw (4,-\myb) node[above=-1pt] {$\updownarrow$};
 \draw (5,-\myb) node[above=-1pt] {$\updownarrow$};

 \bond{0,-\myb};
 \bond{1,-\myb};
 \bond{2,-\myb};
 \tdots{3,-\myb};
 \bond{4,-\myb};

 \DDnode{x}{0,\myb}{};
 \DDnode{w}{1,\myb}{};
 \DDnode{s}{2,\myb}{};
 \DDnode{w}{3,\myb}{};
 \DDnode{w}{4,\myb}{};
 \DDnode{s}{5,\myb}{};

 \DDnode{x}{0,-\myb}{};
 \DDnode{w}{1,-\myb}{};
 \DDnode{s}{2,-\myb}{};
 \DDnode{w}{3,-\myb}{};
 \DDnode{w}{4,-\myb}{};
 \DDnode{s}{5,-\myb}{};

 \draw (0,-\myb) node[below=2pt] {};
 \draw (1,-\myb) node[below=2pt] {};
 \draw (2,-\myb) node[below=2pt] {};
 \draw (3,-\myb) node[below=2pt] {};
 \draw (4,-\myb) node[below=2pt] {};
 \draw (5,-\myb) node[below=2pt] {};

 \end{tikzpicture}
 \]
 determines the shape of the annihilator as listed in Table \ref{F:ann}.  Now express the weight in terms of the simple roots $\alpha_j = \epsilon_j - \epsilon_{j+1}$, where $\epsilon_j$ are the functionals that extract the $j$-th diagonal element of $\mathfrak{sl}(n+1,\C)$:
 \begin{align*}
 -\mu &= \lambda_1 + \lambda_n - 3\overline{\lambda_1} + \overline{\lambda_3} = \alpha_1 + ... + \alpha_n - 2\overline{\alpha_1} - \overline{\alpha_2}\\
 &= \epsilon_1 - \epsilon_{n+1} - 2\overline{\epsilon_1} + \overline{\epsilon_2} + \overline{\epsilon_3}
 \end{align*}
 This determines the remaining condition on the annihilator.
 \end{examp}

Using \eqref{E:SA}, we compute each $\mathfrak{U}_\mu$ and obtain the dimensions listed for $\mathfrak{S}_\mu$ in Theorem \ref{Thm2},
except for type I, $n=2$ for which $\mathfrak{U}_{{\rm I}} = 8$.

In the type I case, the Cartan geometry is equivalent to a complex parabolic geometry of type $(G,P)$ (where these are regarded as complex Lie groups), whose underlying structure is a holomorphic projective structure.  These submaximal symmetry dimensions were classified in \cite{KT} (see \cite{E$_1$} for the real projective case). In terms of the $\C$-dimension $n$ of the underlying complex manifold, we have $\fS_\C = n^2 - 2n + 5$ when $n \geq 3$ and $\fS_\C = 3$ when $n=2$.  Regarded as a c-projective structure of type I, these complex dimensions simply double to get the corresponding real dimensions.
The $n=2$ case is the well-known exception that is the holomorphic analogue of 2-dimensional
projective structures, see \cite[\S4.3]{KT}. This finishes the type I case.

It remains to show that $\fS_{{\rm II}} = \fU_{{\rm II}}$ and $\fS_{{\rm III}} = \fU_{{\rm III}}$.  This is accomplished in Section \ref{S3} by exhibiting type II and type III models whose c-projective symmetries realize the calculated upper bounds.  Alternatively, we now give an abstract proof of realizability via the same technique as used in \cite[\S4.1]{KT}.  There, an abstract model of the regular normal parabolic geometry
was constructed by a deformation idea.
Here, fix a type and consider the (graded) Tanaka algebra $\fa := \fa^{\psi} = \g_{-}\oplus\fa_0^{\psi} \subset\g$, where $\psi = \phi_0 + \overline{\phi_0}$.  Define $\ff = \fa$ as vector spaces, and consider the deformed bracket $[\cdot,\cdot]_{\ff} = [\cdot,\cdot]_\fa - \psi(\cdot,\cdot)$ by regarding $\psi$ as a 2-cochain.  From Table \ref{F:Kh-lw}, we see that with the exception of the $n=2$ type I case, $\psi$ has image in $\g_- \subset \fa$, so $[\cdot,\cdot]_{\ff}$ is well-defined.  As in \cite[Lemma 4.1.1]{KT}, the Jacobi identity on $\ff$ reduces to:
 \begin{align} \label{E:Jac}
\op{Jac}_{\ff}(x,y,z) = \psi(\psi(x,y),z) + \psi(\psi(y,z),x) + \psi(\psi(z,x),y).
 \end{align}
From Table \ref{F:Kh-lw}, for $n \geq 3$, the output of $\psi$ does not depend on any of the root spaces involved in the input to the 2-cochain $\psi$.  Hence, by \eqref{E:Jac}, the Jacobi identity holds for $\ff$.  For $n=2$, this argument works only in the type II case.  In all valid cases, i.e.\ when $\ff$ is a Lie algebra, $\ff/\fa_0$ integrates to a {\em local} homogeneous space $M = F / A_0$.  This space will support a c-projective structure whose symmetry algebra is isomorphic to $\ff$.  This is asserted by an extension functor argument \cite{CS}: Consider the principal $P$-bundle $\mathcal{G} = F \times_{A_0} P \to M$.  An $\ff$-invariant Cartan connection of type $(G,P)$ is determined by the algebraic data of a linear map $\varphi : \ff \to \g$ that is $\fa_0$-equivariant and satisfies $\varphi|_{\fa_0} = \operatorname{id}_{\fa_0}$.  Using the vector space identification $\ff = \fa$, consider the $\ff$-invariant Cartan connection determined by $\varphi = \operatorname{id}_{\ff}$.  Its full curvature corresponds to
 \[
 [\varphi(x),\varphi(y)]_\fa - \varphi([x,y]_\ff) = [x,y]_\fa - [x,y]_\ff = \psi(x,y), \qquad \forall x,y \in \g_-.
 \]
 Hence, the curvature is purely harmonic.  Thus, we have constructed a regular normal Cartan geometry of type $(G,P)$, and its underlying structure is a c-projective geometry of the given type.

 When $n=2$, the above argument fails for:
 \begin{itemize}
 \item type I: the deformation by $\psi$ is not well-defined.  As remarked earlier, $\fS_{{\rm I}} = 6 < \fU_{{\rm I}} = 8$.
 \item type III: the Jacobi identity fails for $\ff$.  However, a different deformation of $\fa_0$ is possible, and a model is given in Section \ref{S3}; see the details in Appendix \ref{S.A}.
Thus, $\fS_{{\rm III}} = \fU_{{\rm III}} = 8$.
 \end{itemize}
This concludes the proof of Theorem \ref{Thm2}.

\section{Submaximal models in the three curvature types}\label{S3}

Let us specify explicit models\footnote{We re-numerate some indices in the matrix models
$\mathfrak{a}_0$ of Section \ref{S2}, e.g.\ $2\leftrightarrow n$ for type II etc. This helps seeing
the stabilization of the models.} realizing the universal bounds for non-flat c-projective structures of
the pure curvature types $\mathbb{V}_i$. We do not claim unicity of these models at this point, they
only prove sharpness of the estimates.

\smallskip

{\bf Type I, ${\mathbf{n>2}}$.} This is the holomorphic version of the Egorov's symmetric connection
\cite{E$_1$}, given in the real case by the Christoffel symbols $\Gamma^1_{23}=x^2$, $\Gamma^i_{jk}=0$ else,
in coordinates $(x^1,\dots,x^n)$ on $\R^n$.

The holomorphic version $\nabla$ is given by
(to get real structures we add complex conjugate to all equations, shortening this to +Cc)
 $$
\Gamma^1_{23}=z^2\ \text{ (+Cc: }\Gamma^{\bar1}_{\bar2\bar3}=\overline{z^2})
 $$
(and $\Gamma^l_{jk}=0$ for all other barred/un-barred indices)
in the coordinate system $(z^1,\dots,z^n)$ on $\C^n$. The complex structure is standard $J=i$.

In the real coordinates $(\tilde x^1,\dots,\tilde x^{2n})$ given by
$z^k=\tilde x^{2k-1}+i \tilde x^{2k}$ we have the following non-zero Christoffel indices:
 $$
\Gamma^{\tilde1}_{\tilde3\tilde5}=\Gamma^{\tilde2}_{\tilde3\tilde6}=\Gamma^{\tilde2}_{\tilde4\tilde5}=
-\Gamma^{\tilde1}_{\tilde4\tilde6}=\tilde x^3, \quad
\Gamma^{\tilde2}_{\tilde3\tilde5}=-\Gamma^{\tilde1}_{\tilde3\tilde6}=-\Gamma^{\tilde1}_{\tilde4\tilde5}=
-\Gamma^{\tilde2}_{\tilde4\tilde6}=\tilde x^4.
 $$
The harmonic curvature has type I and is non-zero: $\kappa_H=\kappa_{{\rm I}}\ne0$. Indeed, the curvature of
$\nabla$ in the complex coordinates is equal to
 $$
W_\nabla=dz^2\we dz^3\ot z^2\p_{z^1}\ \text{ (+Cc)}.
 $$

The c-projective symmetries are found from the equations specified in \S\ref{S1} to be the
real and imaginary parts of the following (linearly independent) holomorphic vector fields:
  \begin{gather*}
\p_{z^1},\quad \p_{z^3},\ \dots,\ \p_{z^n}, \quad z^i\p_{z^j}\quad (i\ge2,\ j\ne2,3),\\
2z^1\p_{z^1}+z^2\p_{z^2},\quad z^1\p_{z^1}+z^3\p_{z^3},\quad z^2z^3\p_{z^1} - \p_{z^2},\quad
(z^2)^3\p_{z^1} - 3z^2\p_{z^3}.
 \end{gather*}
Since the totality of these $2\cdot(n^2-2n+5)$ coincides with the universal upper bound, these
are all symmetries, and so the above $(J,[\nabla])$ is a sub-maximal
c-projective structure of curvature type I.

\smallskip

{\bf Type I, ${\mathbf{n=2}}$.} Real projective structures on $\R^2$ were studied by Lie and Liouville
\cite{Lio}, and Tresse \cite{Tr} classified submaximal projective connections
(in retrospective, as this notions was introduced later by
Cartan \cite{C}; Tresse studied the corresponding 2nd order ODEs).
Complexification yields a submaximal c-projective structure with respect to the
standard complex structure $J=i$ on $\C^2$: $\nabla$ is the complex connection
with the non-zero Christoffel symbols
 $$
\Gamma^1_{22}=-\Gamma^1_{11}=\frac1{2z^1}\ \text{ (+Cc).}
 $$
The c-projective symmetries are real and imaginary parts of the holomorphic fields
(altogether 6 symmetries)
 $$
\p_{z^2},\ z^1\p_{z^1}+z^2\p_{z^2},\ z^1z^2\p_{z^1}+\tfrac12(z^2)^2\p_{z^2}.
 $$

\smallskip

{\bf Type II.} Consider the complex connection $\nabla$ with respect to the standard
complex structure $J=i$ on $\C^n$ given in the complex coordinates $(z^1,\dots,z^n)$
by non-zero Christoffel symbols
 \begin{equation}\label{subCmax}
\Gamma_{11}^2=\overline{z^1}\quad\text{(+Cc: } \Gamma_{\bar1\bar1}^{\bar2}=z^1).
 \end{equation}
Its curvature has pure type II, $\kappa_H=\kappa_{{\rm II}}\ne0$:
 $$
W_\nabla=dz^1\we d\overline{z^1}\otimes z^1\p_{z^2}\ \text{ (+Cc)}.
 $$
The c-projective symmetries are found from the equations specified in \S\ref{S1} to be the
real and imaginary parts of the following (linearly independent) complex-valued vector fields:
  \begin{gather*}
\p_{z^2},\ \dots,\ \p_{z^n}, \quad z^i\p_{z^j}\quad (i\ne2,\ j>1),\\
z^1\p_{z^1}+2z^2\p_{z^2}+\overline{z^2}\p_{\overline{z^2}},\quad
\p_{z^1}-\tfrac12(\overline{z^1})^2\p_{\overline{z^2}}.
 \end{gather*}
Since the totality of these $2(n^2-n+2)$ coincides with the universal upper bound, these
are all symmetries, and so the above $(J,[\nabla])$ is a sub-maximal
c-projective structure of curvature type II.

\smallskip

{\bf Type III, ${\mathbf{n>2}}$.} Sub-maximal symmetric almost complex structures on $\C^3$,
i.e.\ maximally symmetric (measured via functional dimension and rank) among all non-integrable $J$,
were classified in \cite{K$_4$}. There are two different such structures, but only one of them
has the Nijenhuis tensor, given by the lowest weight vector from the Kostant's Borel-Bott-Weil theorem.
Namely, the structure in complex coordinates $(z^1,z^2,z^3)$ is given by
 $$
J\p_{z^1}=i\p_{z^1}+z^2\p_{\overline{z^3}},\ J\p_{z^2}=i\p_{z^2},\ J\p_{z^3}=i\p_{z^3}.
 $$
Let us extend it to $\C^n$ by multiplying with $(\C^{n-3},i)$, i.e.\ letting $J\p_{z^k}=i\p_{z^k}$ for
$k>3$.

In the real coordinates $(\tilde x^1,\dots,\tilde x^6)$ given by
$z^k=\tilde x^{2k-1}+i \tilde x^{2k}$ we have:
 \begin{gather*}
J\p_{\tilde x^1}=\p_{\tilde x^2}+\tilde x^3\p_{\tilde x^5}-\tilde x^4\p_{\tilde x^6},\
J\p_{\tilde x^2}=-\p_{\tilde x^1}-\tilde x^4\p_{\tilde x^5}-\tilde x^3\p_{\tilde x^6},\\
J\p_{\tilde x^{2k-1}}=\p_{\tilde x^{2k}},\ J\p_{\tilde x^{2k}}=-\p_{\tilde x^{2k-1}}\ (k>1).
 \end{gather*}

To find a complex connection for this $J$, let $\tilde\nabla=d$ be the trivial connection,
i.e.\ its Christoffel symbols vanish in the given coordinate system. Then
$\nabla=\frac12(\tilde\nabla-J\tilde\nabla J)$ is a complex connection. Its non-zero
Christoffel symbols are
 $$
\Gamma_{21}^{\bar3}=\frac{i}2\ \text{ (+Cc: } \Gamma_{\bar2\bar1}^3=-\frac{i}2).
 $$
In real coordinates these write so:
 $$
\Gamma_{\tilde3\tilde1}^{\tilde6}=\Gamma_{\tilde3\tilde2}^{\tilde5}=
\Gamma_{\tilde4\tilde1}^{\tilde5}=-\Gamma_{\tilde4\tilde2}^{\tilde6}=-\frac12,
 $$
and, in particular, $\Gamma_{\tilde1\tilde3}^{\tilde6}=\Gamma_{\tilde2\tilde3}^{\tilde5}=
\Gamma_{\tilde1\tilde4}^{\tilde5}=\Gamma_{\tilde2\tilde4}^{\tilde6}=0$.

Thus the torsion $T_\nabla\neq0$, while the curvature $R_\nabla=0$.
In fact,
 $$
N_J=-2i\,dz^1\we dz^2\ot\p_{\overline{z^3}}\ \text{ (+Cc)}.
 $$
Consequently, $\kappa_H=\kappa_{\rm{III}}\neq0$,
the connection $\nabla$ is minimal and the c-projective structure $(J,[\nabla])$ has curvature type III.

The c-projective symmetries are found from the equations specified in \S\ref{S1} to be the
real and imaginary parts of the following (linearly independent) complex-valued vector fields:
  \begin{gather*}
\p_{z^1},\quad \p_{z^3},\ \dots,\ \p_{z^n}, \quad z^i\p_{z^j}\quad (i\ne3,\ j>2),\\
z^1\p_{z^1}+\overline{z^3}\p_{\overline{z^3}},\quad
z^2\p_{z^2}+\overline{z^3}\p_{\overline{z^3}},\quad
\p_{z^2}+\tfrac{z^1}{2i}(\p_{z^3}+\p_{\overline{z^3}}),\\
z^1\p_{z^2}-\tfrac{i}4(z^1)^2\p_{\overline{z^3}},\quad
z^2\p_{z^1}-\tfrac{i}4(z^2)^2\p_{\overline{z^3}}.
 \end{gather*}
Since the totality of these $2\cdot(n^2-2n+6)$ coincides with the universal upper bound, these
are all symmetries, and so the above
$(J,[\nabla])$ is a sub-maximal c-projective structure of curvature type III.

\smallskip

{\bf Type III, ${\mathbf{n=2}}$.}
This is the exceptional case, for which the method of Section \ref{S2} does not give even
an abstract model. However the abstract bound $\dim\le8$ is sharp.
We provide the local model $(M^4,J,[\nabla])$ in real coordinates $(x,y,p,q)$.
The almost complex structure in the basis $e_1=\p_x$, $e_2=\p_y$, $e_3=\p_p$,
$e_4=\p_q-\frac{3y}{2p}\,\p_x-\frac{5x}{2p}\,\p_y$ is given by:
 $$
Je_1=e_2,\ Je_2=-e_1,\ Je_3=e_4,\ Je_4=-e_3.
 $$
In the dual co-basis $\theta_1=dx+\frac{3y}{2p}\,dq$, $\theta_2=dy+\frac{5x}{2p}\,dq$, $\theta_3=dp$, $\theta_4=dq$
the minimal complex connection is given by:
 \begin{alignat*}{2}
\nabla e_1=\frac1{2p}\,e_2\ot\theta_4,\quad &
\nabla e_3=-\frac1{p}\,(e_1\ot\theta_1-e_2\ot\theta_2+e_3\ot\theta_3+e_4\ot\theta_4)\\
 & \ -\frac1{4p^2}\,\bigl(e_1\ot(3x\theta_3+3y\theta_4)+e_2\ot(3y\theta_3+13x\theta_4)\bigr)
 \end{alignat*}
(these relations are enough since $\nabla Je_k=J\nabla e_k$).

The torsion $T_\nabla$ is non-zero and represents $\kappa_{\op{III}}$. The curvature $R_\nabla$ is however also non-zero
and has type $(1,1)$. One could suspect that this yields a harmonic curvature of type II ($\kappa_{\op{I}}=0$ because the (2,0)-part
of $R_\nabla$ vanishes, and also because elsewise the dimension of $\mathfrak{cp}(\nabla,J)$ would be bounded by 6),
but the structure equations show $\kappa_{\op{II}}=0$ (we delegate the details of this computation to the appendix).

Thus the above pair $(J,[\nabla])$ has harmonic curvature of pure type III. And its symmetry algebra has dimension 8, here are
the generators:
 \begin{gather*}
x\p_x+y\p_y,\ p^{-3/2}\p_y,\ p\p_p+q\p_q,\ \p_q,\\
p\,(y\p_x-x\p_y)-2pq\,\p_p+(p^2-q^2)\p_q,\ \frac{p^2+q^2}{p^{3/2}}(p\p_x-q\p_y),\\
\frac{p^2+q^2}{2p^{3/2}}\p_y+\frac{q}{p^{3/2}}(q\p_y-p\p_x),\ p^{-3/2}(q\p_y-\tfrac13p\p_x)
 \end{gather*}
(notice that all the symmetries are actually affine).

This finishes realization (by models) of the universal bounds on submaximal symmetry dimensions.

 \begin{rk}\label{trans}
As a consequence of realization and the results of \cite{KT}, the equality $\fU=\fS$ yields local transitivity (around any regular point) for the
symmetry algebra of any c-projective structure with submaximal symmetry (this also applies to general c-projective structures considered
in the next section), as well as for any c-projective structure of fixed curvature type with the submaximal symmetry dimension $\fS_i$.
 \end{rk}

\section{C-projective structures: the general case.}\label{S4}

In this section we encode general (not necessary minimal) c-projective structures
as real (not necessary normal) parabolic geometries of type $\op{SL}(n+1,\C)_\R/P$.

Assume at first that $\pi:\mathcal{G}\to M$ is a principal $P$-bundle with a
Cartan connection $\omega=\omega_{-1}+\omega_0+\omega_1\in\Omega^1(\mathcal{G},\g)$.
The covariant derivative can be read off the Cartan connection of any $G_0$-reduction of this bundle
(Weyl structure) by the following formula, cf. \cite[Proposition 1.3.4]{CS}:
 $$
\nabla_XY=d\pi\circ\omega_{-1}^{-1}\bigl(\tilde X\cdot\omega_{-1}(\tilde Y)
-\omega_0(\tilde X)(\omega_{-1}(\tilde Y))\bigr),
 $$
where $\tilde X,\tilde Y$ are arbitrary lifts of $X,Y\in\mathcal{D}(M)$ to vector fields on $\mathcal{G}$
(independence of the lift for $Y$ is obvious, for $X$ follows from the equivariance of $\omega$;
one also checks independence of the point $a\in\pi^{-1}(x)$).

Since the first frame bundle reduction, driven by the almost complex structure $J$, forces $\omega_{-1}$
to be a complex isomorphism between $(T_xM,J)$ and $(\C^n,i)$ and since $\omega_0$ takes values in
$\mathfrak{gl}(n,\C)$, we conclude that $\nabla_XJY=J\nabla_XY$, i.e.\ $\nabla J=0$. Thus
parabolic geometries encode c-projective geometries with classes of complex connections
$\nabla$ only.

Moreover, a choice of connection $\nabla$ in a c-projective class with fixed (normalized) torsion
corresponds to a Weyl structure of $(\mathcal{G},\omega)$ and a change of this gives
a c-projective change of $\nabla$. Thus we have to show only that a c-projective class of a complex
connection $\nabla$ with a fixed torsion can be represented as a parabolic geometry.

To begin with let us see how we can modify the torsion keeping the class of $J$-planar curves fixed.
By \cite[Appendix A]{K$_1$} a (2,1)-tensor decomposes into $J$-linear/antilinear components
as follows
 $$
T_\nabla=T_\nabla^{++}+T_\nabla^{+-}+T_\nabla^{-+}+T_\nabla^{--},
 $$
where $T^{\epsilon_1,\epsilon_2}_\nabla(J^{k_1}X,J^{k_2}Y)=\epsilon_1^{k_1}\epsilon_2^{k_2}
J^{k_1+k_2}T^{\epsilon_1,\epsilon_2}_\nabla(X,Y)$.

The component $T^{++}$ is killed similar to the real case: $\nabla\simeq\nabla-\frac12T^{++}_\nabla$.
The component $T^{--}_\nabla=\frac14N_J$ is invariant \cite{Lic}. Next
$T_\nabla^{+-}=-T_\nabla^{-+}\circ\tau$, where
$\tau:\Lambda^2V^*\ot V\to\Lambda^2V^*\ot V$ ($V=TM$) is swap of the first two arguments,
so it is enough to treat $T_\nabla^{-+}$. Again by \cite{K$_1$} the gauge is
 $$
T_\nabla^{-+}\mapsto T_\nabla^{-+} - A,
 $$
where $A\in\Lambda^2V^*\ot V$ is antilinear-linear $(2,1)$-tensor,
which means $A(JX,Y)=-A(X,JY)=-JA(X,Y)$ $\forall X,Y\in V$.

 \begin{lem}\label{L1}
An antilinear-linear $(2,1)$-tensor $A$ satisfying the property $A(X,X)\in\C\cdot X$ has the
following form for some 1-form $\vp$:
 $$
A(X,Y)=\vp(X)Y+\vp(JX)JY.
 $$
 \end{lem}

 \begin{proof} By polarization $A(X,X)\in\C\cdot X$ implies
 $$
A(X,Y)+A(Y,X)\in\C\cdot\langle X,Y\rangle.
 $$
Substitution $Y\mapsto JY$ and use of $\C$-linearity/antilinearity (and cancelation of $J$) yields $A(X,Y)-A(Y,X)\in\C\cdot\langle X,Y\rangle$. Hence $A(X,Y)\in\C\cdot\langle X,Y\rangle$,
i.e. $A(X,Y)=\Phi(X)Y+\Psi(Y)X$ for some $\C$-valued 1-forms $\Phi,\Psi$.
Antilinearity/linearity by $X,Y$ resp.\ yields $\Psi(Y)=0$, and $\Phi(X)=\vp(X)+i\,\vp(JX)$.
 \end{proof}

Denote the space of $A$-tensors from Lemma \ref{L1} by $\mathbb{T}^{-+}_\text{trace}$.
This is a submodule of $\mathbb{T}^{-+}$ part of decomposition of the space of torsion tensors module
$\mathbb{T}=\Lambda^2V^*\ot V$, decomposed into irreducible $\op{GL}(n,\C)$-submodules as follows
(the part $\mathbb{T}^{+-}\simeq\mathbb{T}^{-+}$ appears with
its swap, so only one part enters the decomposition)
 $$
\mathbb{T}= \mathbb{T}^{++}_\text{trace}\oplus\mathbb{T}^{++}_\text{traceless}
\oplus\mathbb{T}^{--}\oplus\mathbb{T}^{-+}_\text{traceless}\oplus\mathbb{T}^{-+}_\text{trace}.
 $$
By traceless we mean that the endomorphism obtained by filling one argument of the $(2,1)$-tensor
is traceless (and trace is the invariant complement). In complexification this decomposition writes\footnote{We 
abbreviate $\Lambda^{p,q}=\Lambda^{p,q}V^*$ and similar for $S^{p,q}$ throughout the text.}
 \begin{multline*}
\mathbb{T}^\C=(\La^{2,0}\ot V_{1,0}+\op{Cc})_\text{trace}
\oplus(\La^{2,0}\ot V_{1,0}+\op{Cc})_\text{traceless}\oplus\\
(\La^{0,2}\ot V_{1,0}+\op{Cc})
\oplus(\La^{1,1}\ot V_{1,0}+\op{Cc})_\text{traceless}
\oplus(\La^{1,1}\ot V_{1,0}+\op{Cc})_\text{trace},
 \end{multline*}
where $\op{Cc}$ stays for complex-conjugate as before.

Let us denote the projection to part $k$ in the above decomposition by $\pi_k$. From the discussion above
we can kill $\pi_1(T_\nabla)$, $\pi_2(T_\nabla)$, $\pi_5(T_\nabla)$ by a choice of representative within the
c-projective class of $\nabla$, but the components $\pi_3(T_\nabla)=\kappa_\text{III}$ and $\pi_4(T_\nabla)$ are invariant.

 \begin{cor}\label{c3}
The c-projective class $[\nabla]$ contains a minimal $J$-complex connection $\nabla$ iff $\pi_4(T_\nabla)=0$. \qed
 \end{cor}

Remaining freedom in choosing $\nabla$ is given by the standard formula:

 \begin{lem}\label{L2}
Two $J$-complex connections $\nabla,\bar\nabla$ with vanishing parts 1,2,5 of the torsion
are c-projectively equivalent iff
 $$
\bar\nabla_XY=\nabla_XY+\Psi(X)Y+\Psi(Y)X-\Psi(JX)JY-\Psi(JY)JX
 $$
for some 1-form $\Psi\in\Omega^1(M)$ (notice that $T_\nabla=T_{\bar\nabla}$).
 \end{lem}

 \begin{proof} Indeed, the tensor $A=\bar\nabla-\nabla$ satisfies $A^{-}=0$ and $A^{+}=A^{+}\circ\tau$
in terms of \cite{K$_1$}.
 \end{proof}

Let us denote $\varrho=\pi_1+\pi_2+\pi_5$, so that the assumption of the Lemma is $\varrho(T_\nabla)=0$.
Since this tensorial projection is applicable to the lowest part of the curvature $\kappa$ of the Cartan connection $\omega$,
viewed as $P$-equivariant function, we rewrite the equality as $\varrho(\kappa_1)=0$, where $\kappa_i$ is the part of
the curvature $\kappa$ of $\fg_0$-homogeneity $i$. 

Recall \cite{CS} that the Kostant codifferential $\p^*:\Lambda^i\g_-^*\ot\g\to\Lambda^{i-1}\g_-^*\ot\g$
is adjoint to the Lie algebra differential $\p:\Lambda^i\mathfrak{g}_-^*\ot\g\to\Lambda^{i+1}\mathfrak{g}_-^*\ot\g$.

Consider the curvature module $\mathbb{U}=\Lambda^2\mathfrak{g}_-^*\ot\g$ over
$\g_0=\mathfrak{gl}(n,\C)$, and decompose it into the graded parts
$\mathbb{U}=\mathbb{U}_1\oplus\mathbb{U}_2\oplus\mathbb{U}_3$. For $\mathbb{U}_2$ denote the
sum of real $\g_0$-irreps by $\mathbb{U}_2^r$ and the sum of complex $\g_0$-irreps by $\mathbb{U}_2^c$ 
(there are no quaternionic parts).

Let us define $\mathfrak{C}_1=\op{Ker}(\varrho)\subset\mathbb{U}_1$, $\mathfrak{C}_2=(\op{Ker}(\p^*)\cap\mathbb{U}_2^c)\oplus(\mathfrak{p}_+\cdot\mathfrak{C}_1\cap\mathbb{U}_2^r)$,
and also $\mathfrak{C}_3=\mathbb{U}_3$. Then $\mathfrak{C}=\mathfrak{C}_1\oplus\mathfrak{C}_2\oplus\mathfrak{C}_3$ \
is a $G_0$-submodule of $\mathbb{U}$.

 \begin{prop}\label{P3}
The submodule $\mathfrak{C}$ is $P$-invariant.
 \end{prop}

 \begin{proof}
It suffices to check $\mathfrak{p}_+$-invariance. 
Let us decompose the graded parts of $\mathbb{U}=\mathbb{U}_1\oplus\mathbb{U}_2\oplus\mathbb{U}_3$ into
$\g_0$-irreducibles. For $\mathbb{U}_1=\mathbb{T}^\C$ this was done before Corollary \ref{c3}, whence 
 $$
\mathfrak{C}_1=\mathbb{T}^{--}\oplus\mathbb{T}^{-+}_\text{traceless}
 $$
with the complexification\footnote{Recall that $\odot$ denotes the Cartan product
(the same as previous "traceless").}
 $$
\mathfrak{C}_1^\C=(\Lambda^{0,2}\ot V_{1,0})\oplus(\Lambda^{1,1}\odot V_{1,0})+\op{Cc}.
 $$
Next we decompose $\mathbb{U}_2$ into $\g_0$-irreducibles. For $n\ge4$ we have:
 \com{
In complexification 
 $$
(\Lambda^2\g_-^*\otimes\g_0)^\C=
(\Lambda^{2,0}+\Lambda^{1,1}+\Lambda^{0,2})\otimes(\Lambda^{1,0}\otimes V_{1,0}) + \operatorname{Cc}
 $$
and we compute for $n\ge4$:
 \begin{align*}
 \Lambda^{2,0} \otimes \Lambda^{1,0} \otimes V_{1,0} &= ( \Lambda^{2,0} \odot \Lambda^{1,0} \odot V_{1,0}) \oplus 2 \Lambda^{2,0} 
\oplus S^{2,0}\\ 
&\hspace{112pt} \oplus (\Lambda^{3,0} \odot V_{1,0}) + \operatorname{Cc},\\
 \Lambda^{1,1} \otimes \Lambda^{1,0} \otimes V_{1,0} &= (\Lambda^{2,1} \odot V_{1,0}) \oplus (S^{2,1} \odot V_{1,0}) \oplus 2\Lambda^{1,1}+ \operatorname{Cc},\\
 \Lambda^{0,2} \otimes \Lambda^{1,0} \otimes V_{1,0} &= (\Lambda^{1,2} \odot V_{1,0}) \oplus \Lambda^{0,2} + \operatorname{Cc}.
 \end{align*}
Hence, we get
 \begin{align*}
(\Lambda^2\g_-^*\otimes\g_0)^\C&=[(\Lambda^{2,0}\odot\Lambda^{1,0}\odot V_{1,0})\oplus
3\,\Lambda^{2,0}\oplus S^{2,0}\oplus(\Lambda^{3,0}\odot V_{1,0}) \\
&\hspace{-20pt}\oplus(\Lambda^{2,1}\odot V_{1,0})\oplus
(S^{2,1}\odot V_{1,0})\oplus (\Lambda^{1,2}\odot V_{1,0}) + \operatorname{Cc}] \oplus 4\,\Lambda^{1,1}.
 \end{align*}
 }
 \begin{multline*}
\!\!(\Lambda^2\g_-^*\ot\g_0)^\C=
(\Lambda^{2,0}\oplus\Lambda^{1,1}\oplus\Lambda^{0,2})\otimes(\Lambda^{1,0}\odot V_{1,0}+\C) + \op{Cc}\\
=[(\Lambda^{2,0}\odot\Lambda^{1,0}\odot V_{1,0})\oplus
3\,\Lambda^{2,0}\oplus S^{2,0}\oplus(\Lambda^{3,0}\odot V_{1,0})\oplus \\
(\Lambda^{2,1}\odot V_{1,0})\oplus
(S^{2,1}\odot V_{1,0})\oplus (\Lambda^{1,2}\odot V_{1,0}) + \op{Cc}] \oplus 4\,\Lambda^{1,1}.
 \end{multline*}
Every term in $[\dots]$ together with its complex conjugate gives a real irreducible submodule
($\Lambda^{1,1}$ is already real irreducible)
that we denote successively (not counting multiplicity) by $\mathbb{K}_1,\dots,\mathbb{K}_8$, and we get:
 $$
\mathbb{U}_2= 
\underbrace{\mathbb{K}_1\oplus3\mathbb{K}_2\oplus\mathbb{K}_3\oplus\mathbb{K}_4\oplus\mathbb{K}_5
\oplus\mathbb{K}_6\oplus\mathbb{K}_7}_{\mathbb{U}_2^c}
\oplus \underbrace{4\mathbb{K}_8}_{\mathbb{U}_2^r}.
 $$
We will not need the decomposition of $\mathbb{U}_3$.

The module $\mathbb{U}_2$ is mapped by $\p^*$ onto the module
 $$
(\Lambda^1\g_-^*\ot\g_1)^\C=
[S^{2,0}\oplus\Lambda^{2,0} +\op{Cc}] \oplus 2\,\Lambda^{1,1}=
(\mathbb{K}_2\oplus\mathbb{K}_3\oplus2\mathbb{K}_8)^\C.
 $$
Consequently we get the following decomposition into $\g_0$-irreps:
 \begin{equation}\label{p*2}
\op{Ker}(\p^*)\cap\mathbb{U}_2=
\mathbb{K}_1\oplus2\mathbb{K}_2\oplus\mathbb{K}_4\oplus\mathbb{K}_5\oplus\mathbb{K}_6
\oplus\mathbb{K}_7\oplus2\mathbb{K}_8.
 \end{equation}
The action of $\mathfrak{p}_+$ is trivial on the first factor of $\La^2\g_-^*\ot\g$, and so
the restricted action on $\mathfrak{C}$ maps $\mathbb{T}^{--}$ to $\mathbb{K}_2\oplus\mathbb{K}_7$ --
notice that this $\mathbb{K}_2$ belongs to $\op{Ker}(\p^*)$ by $\mathfrak{p}_+$-equivariance of $\p^*$,
so it is one of the terms in \eqref{p*2}.
Also, $\mathfrak{p}_+$ maps $\mathbb{T}^{-+}_\text{traceless}$ to $\mathbb{K}_5\oplus\mathbb{K}_6\oplus2\mathbb{K}_8$,
but the latter (double) term differs from the similarly named terms in \eqref{p*2}.
To distinguish these denote
$2\tilde{\mathbb{K}_8}=(\mathfrak{p}_+\cdot\mathfrak{C}_1)\cap4\mathbb{K}_8\subset\mathbb{U}_2$
($\op{Ker}(\p^*)\cap2\tilde{\mathbb{K}_8}=0$). Then 
 \begin{multline*}
\mathfrak{C}_2= (\op{Ker}(\p^*)\cap(\mathbb{K}_1\oplus3\mathbb{K}_2\oplus\dots\oplus\mathbb{K}_7))\oplus
((\mathfrak{p}_+\cdot\mathfrak{C}_1)\cap4\mathbb{K}_8)\\
=\mathbb{K}_1\oplus2\mathbb{K}_2\oplus\mathbb{K}_4\oplus\mathbb{K}_5\oplus\mathbb{K}_6
\oplus\mathbb{K}_7\oplus2\tilde{\mathbb{K}_8}.
 \end{multline*}
Now $\mathfrak{p}_+$ maps $\mathfrak{C}_1$ to $\mathfrak{C}_2$, and obviously it maps the latter to
$\mathfrak{C}_3$. Therefore we conclude invariance with respect to $P=G_0\ltimes\mathfrak{p}_+$ for $n\ge4$.

For $n=3$ we have $\mathbb{K}_2=\mathbb{K}_4$ as $A_2$-modules (ignoring $\mathfrak{z}(\g_0)$).
So only one multiplicity changes in the decomposition of $\mathfrak{C}_2$. For
$n=2$ more terms in the above decompositions change/disappear,
but the arguments persist and the conclusion is not altered.
 \end{proof}

 \begin{rk}
A normalization different from the standard $\p^*\kappa=0$ was used previously by D.Fox in \cite{F}
(our normalization differs from his).
 \end{rk}

 \begin{lem}\label{LtrC2}
The subspace $\mathfrak{C}_2$ is complementary to $\op{Im}(\p)\subset\Lambda^2\mathfrak{g}_-^*\ot\g_0$.
 \end{lem}

 \begin{proof}
Since $\op{Ker}(\p^*)\cap\mathbb{U}_2$ is complementary to $\op{Im}(\p)\cap\mathbb{U}_2$, it is enough
to show that $(2\tilde{\mathbb{K}_8})\cap\op{Im}(\p)=0$. For this, because $\p^2=0$, it is enough to show
that the map $\p:2\tilde{\mathbb{K}_8}\to\Lambda^3\mathfrak{g}_-^*\ot\g_{-}$ is injective.

An element $\z$ of this module has the form
$\vp_{1,0}\ot\vp_{0,1}\ot(a\epsilon^{1,0}+b\epsilon^{0,1})$,
where $\vp_{1,0}\in\La^{1,0}$ is some element (assume nonzero),
$\epsilon^{1,0}\in\La^{1,0}\ot V_{1,0}$ is the identity,
and similar for $\vp_{0,1},\epsilon^{0,1}$, while $a,b$ are some numbers.

For vectors $u^{1,0},v^{1,0},w^{0,1}\in V^\C=V_{1,0}\oplus V_{0,1}$ of the indicated type
 $$
\p\z(u^{1,0},v^{1,0},w^{0,1})= a\cdot
(\vp_{1,0}(v^{1,0})u^{1,0}-\vp_{1,0}(u^{1,0})v^{1,0})\,\vp_{0,1}(w^{0,1}).
 $$
If this is zero for all choices $u^{1,0},v^{1,0},w^{0,1}$, then $a=0$.
Similarly, substituting $u^{1,0},v^{0,1},w^{0,1}$ we obtain $b=0$.
 \end{proof}

Now comes the main result of this section (which also solves the equivalence problem for general c-projective structures).

 \begin{theorem}
There is an equivalence of categories between c-projective structures $(M,J,[\nabla])$ 
and parabolic geometries of type $\op{SL}(n+1,\C)_\R/P$
with the curvature normalized by the ($P$-invariant) condition $\kappa\in\mathfrak{C}$.
 \end{theorem}

 \begin{proof} Given a pair $(J,[\nabla])$ we first consider the reduction $\mathcal{G}_0$ of the
first frame bundle $\mathcal{F}_M$ corresponding to the choice of $J$. Next we construct the full frame bundle
$\mathcal{G}=\cup_{u\in\mathcal{G}_0}\mathcal{G}_u$, where
 $$
\mathcal{G}_u=\{\theta(u)+\gamma^\nabla(u):\nabla\in[\nabla],
\pi_1(T_\nabla)=\pi_2(T_\nabla)=\pi_5(T_\nabla)=0\}
 $$
and $\theta=\omega_{-1}\in\Omega^1(\mathcal{G}_0,\g_{-1})$ is the soldering form,
$\gamma^\nabla=\omega_0\in\Omega^1(\mathcal{G}_0,\g_0)$ is the principal connection corresponding to
$\nabla$. The topology and the manifold structure on $\mathcal{G}$ is induced naturally (through the Weyl structures \cite{CS}
corresponding to Weyl connections $\nabla$).

By construction, $\mathcal{G}$ is equipped with $G_0$-equivariant 1-form
$\omega_{-1}+\omega_0:T_u\mathcal{G}\to\g_{-1}\oplus\g_0$. We extend it in a
$P$-equivariant way to a Cartan connection $\omega\in\Omega^1(\mathcal{G},\g)$,
however we have to fix the normalization. As usual this is done by the curvature function
$\kappa:\mathcal{G}\to\Lambda^2\g_{-}^*\ot\g$ corresponding to the curvature form
$K=d\omega+\frac12[\omega,\omega]$. Notice that grading 1 part $\kappa_1$ of the curvature, corresponding to
$d\omega_{-1}(\xi,\eta)+[\omega_{-1}(\xi),\omega_0(\eta)]+[\omega_0(\xi),\omega_{-1}(\eta)]$,
does not involve $\omega_1$, while the grading 2 part $\kappa_2$, corresponding to
$d\omega_0(\xi,\eta)+[\omega_{-1}(\xi),\omega_1(\eta)]+[\omega_0(\xi),\omega_0(\eta)]
+[\omega_1(\xi),\omega_{-1}(\eta)]$, is affine in $\omega_1$. 
We can use this fact and Lemma \ref{LtrC2} to construct $\omega_1$ by the condition $\kappa_2\in\mathfrak{C}_2$.

Indeed, the Kostant co-differential is the left-inverse of the  Lie algebra cohomology differential $\p$ (= Spencer differential $\delta$) 
at the indicated place in the complex
 $$
0\to\g_1\ot\g_{-1}^*\stackrel{\p}\longrightarrow\g_0\ot\La^2\g_{-1}^*\to\g_{-1}\ot\La^3\g_{-1}^*\to0.
 $$
The space $\op{Ker}(\p^*)$ is complementary to $\op{Im}(\p)\subset\g_0\ot\La^2\g_{-1}^*\ni\kappa_2$
and since a change of  (Weyl) connection is equivalent to a change of $\kappa_2$ by $\p\psi$, where $\psi\in\g_1\ot\g_{-1}^*$, we can achieve $\kappa_2\in\op{Ker}(\p^*)$ precisely as in the normal case, and $\omega_1$ is (canonically) fixed.

This construction of $\mathcal{G}$ and $\omega=\omega_{-1}+\omega_0+\omega_1$ is clearly
functorial implying the equivalence claim to one side.

To the other side, if we have a Cartan geometry $(\mathcal{G},\omega)$ of the type $\op{SL}(n+1,\C)_\R/P$,
then we read off $J$ from $\mathcal{G}_0$ and sections of $\mathcal{G}\to\mathcal{G}_0$
determine the class of connections $\nabla$. A change of such section is equivalent to a c-projective
change of connection as in Lemma \ref{L2}.
 \end{proof}

 \begin{rk}
It was noticed in \cite{H} that normality of the Cartan connection $\omega$
implies minimality of $\nabla$, i.e.\ $T_\nabla=T_\nabla^{--}=\frac14N_J$ or equivalently
$T_\nabla^{++}=0$, $T_\nabla^{-+}=0$.
On the other hand, for c-projective structures $(J,[\nabla])$ with minimal $\nabla$
references \cite{Y,CEMN}
provide construction of the normal Cartan connection $\omega\in\Omega^1(\mathcal{G},\g)$.
Thus the above equivalence of categories restricts to equivalence of (sub-)categories between
c-projective structures $(J,[\nabla])$ with minimal $\nabla$ and normal parabolic
geometries of type $\op{SL}(n+1,\C)_\R/P$.
 \end{rk}

\section{The general submaximal symmetry dimension.}\label{S5}

Here we derive the submaximal symmetry dimension for general c-projective structures.
In this case, we additionally have the invariant part of the torsion that we denote by 
$\kappa_\text{IV}=\pi_4(T_\nabla)\in\mathbb{T}^{-+}_\text{traceless}$.
This is the obstruction for c-projective geometry to be normal/minimal.

Flatness, i.e.\ local isomorphism to $(\C P^n,J_\text{can},[\nabla^\text{FS}])$, is characterized by:
$\kappa_\text{I}=0,$ $\kappa_\text{II}=0$, $\kappa_\text{III}=0$, $\kappa_\text{IV}=0$.
This system is $P$-invariant, as follows from the proof of Proposition \ref{P3} 
(but its second term $\kappa_\text{II}$ is no longer $P$-invariant: under $\mathfrak{p}_+$ action 
it changes by a derivative of $\kappa_\text{IV}$).

The method developed in \cite{KT} extends to this situation and we shall show that
the bound on submaximal symmetry dimension persists.

\smallskip

 \begin{Proof}{Proof of Theorem \ref{Thm1}}
If $\kappa_\text{IV}=0$, then the connection is minimal and the estimate from above on the submaximal 
symmetry dimension $\mathfrak{S}$ follows from Section \ref{S2}:
 $$
\mathfrak{S}\leq\mathfrak{U}:=\max\{\dim(\fa^\phi)|
0\neq\phi\in\mathbb{V}_\text{I}\oplus\mathbb{V}_\text{II}\oplus\mathbb{V}_\text{III}\}.
 $$

Assume now that $\kappa_\text{IV}$ is a non-zero element in $\mathbb{T}^{-+}_\text{traceless}$.
In this proof this module will be considered as a completely reducible $P$-submodule of
$\La^2(\g/\mathfrak{p})^*\otimes(\g/\mathfrak{p})=\mathbb{U}/(\mathbb{U}_2\oplus\mathbb{U}_3)$, 
so that the $P$-action reduces to the $G_0$-action ($\mathfrak{p}_+$ acts trivially).

Let us notice that normality condition is not crucial for the universal upper bound on
the symmetry dimension in \cite[Section 2.4]{KT}. The essential step is the reduction of $\g_0$ to the
annihilator of the (harmonic) curvature and its Tanaka prolongation, so it can be generalized 
(see also \cite[Theorem~2]{K$_5$}).
Thus, replacing the harmonic curvature with $\kappa_\text{IV}$, this leads to the submaximal symmetry dimension 
$\hat{\mathfrak{S}}$ for general c-projective structures:
 $$
\hat{\mathfrak{S}}\leq\hat{\mathfrak{U}}:=\max\bigl\{\mathfrak{U},\max\{\dim(\fa^\phi)|
0\neq\phi\in\mathbb{T}^{-+}_\text{traceless}\}\bigr\}.
 $$

The complexification of the module $\mathbb{T}^{-+}_\text{traceless}$ is $\mathbb{W}\oplus\overline{\mathbb{W}}$, where 
the lowest weight vector $\phi_0\in\mathbb{W}$ is $e_{\a_1}\we e_{\overline{\a}_1}\ot e_{-\a_1-\dots-\a_n}$. 
The annihilator of $\phi_0+\overline{\phi_0}$ in $\g_0$ is equal to
 \[
\mathfrak{a}_0=\left(\begin{array}{c|cccccc}
 a_0 & 0 & 0 & 0 & \cdots & 0 & 0\\ \hline
 0 & a_1 & 0 & 0 & \cdots & 0 & 0\\
 0 & * & a_2 & * & \cdots & * & 0\\
 0 & * & * & a_3 & \cdots & * & 0\\
 \vdots &  \vdots & \vdots & \vdots & \ddots & \vdots & 0\\
 0 & * & * & * & \cdots & a_{n-1} & 0\\
 0 & * & * & * & \cdots & * & a_n\\
\end{array} \right)
\qquad\quad\begin{matrix}{}\\\\
a_0+\dots+a_n=0\\  {}\\ {}\\ a_1+\overline{a}_1=\overline{a}_0+a_n {}\\ {}\\ {}\end{matrix}
 \]
and so $\dim\mathfrak{a}_0=2(n-1)^2+2$.

In Section \ref{S2}, Propositions \ref{P:lw-vec} and \ref{P:PR} can be applied to $\mathbb{T}^{-+}_\text{traceless}$, i.e.\ 
the normality condition is not essential.  Hence, we have the prolongation rigidity phenomenon, and 
$\mathfrak{a}^{\phi_0}=\g_-\oplus\mathfrak{a}_0$.

Thus the symmetry dimension does not exceed $\dim\mathfrak{a}^{\phi_0}=2n^2-2n+4$.
Since this does not exceed the maximal bound $\mathfrak{U}$ of the three pure curvature types in
Theorem \ref{Thm2}, the conclusion of Theorem \ref{Thm1} follows. \qed
  \end{Proof}

\smallskip

If we assume $N_J=0$, and so eliminate the spontaneous growth of submaximal dimension for $n=3$ (with the winning type III), 
then the submaximal dimension is $2n^2-2n+4$ for all $n\ge2$. This bound persists even in the non-minimal case:

 \begin{prop}
For the general c-projective structure with $\kappa_\textrm{IV}\not\equiv0$
the symmetry dimension does not exceed $2n^2-2n+4$. This bound is sharp.
 \end{prop}

 \begin{proof}
The upper bound follows from the above proof, so we just need to prove realizability, i.e.\ to
construct {\it a model\/}.

We take $J=i$ and use the complex notations. Take $\Gamma^2_{\bar11}=\Gamma^{\bar2}_{1\bar1}=1$ and
all other Christoffel symbols zero (in particular, $\Gamma^2_{1\bar1}=\Gamma^{\bar2}_{\bar11}=0$).
Its torsion $T_\nabla=dz^1\we d\overline{z^1}\ot(\p_{\overline{z^2}}-\p_{z^2})$ has only
$\mathbb{T}^{-+}_\text{traceless}$-component non-zero and its curvature $R_\nabla$ vanishes. 

The c-projective symmetries are the real and imaginary parts of the following (linearly independent)
complex-valued vector fields:
 $$
\p_{z^i},\quad z^i\p_{z^j}\ (i\ne2,j\ne1),\quad z^1\p_{z^1}+z^2\p_{z^2}+\overline{z^2}\p_{\overline{z^2}}.
 $$
The totality of these is $2n^2-2n+4$ as required.
 \end{proof}

\section{C-projective structures: the metric case.}\label{S6}

The goal of this and the next section is to prove Theorem \ref{Thm3}.
In this section we recall the necessary background on metric c-projective structures and derive a useful estimate
involving the degree of mobility; then in the next we give the proof and further specifications.

Two pseudo-K\"ahler metrics $g$ and $\tilde g$ underlying the same complex structure $J$ are
called c-projectively equivalent if their Levi-Civita connections
$\nabla=\nabla^g,\tilde\nabla=\nabla^{\ti g}$ are.

This can be expressed \cite{DM,MS} through the $(1,1)$-tensor
 $$
A={\tilde g}^{-1}g\cdot\left|\frac{\det(\tilde g)}{\det(g)}\right|^{1/2(n+1)}:TM\to TM,
 $$
where ${\tilde g}^{-1}$ is the inverse of ${\tilde g}$ ($\tilde g^{ik}\bar g_{kj}=\delta^i_j$),
and ${\tilde g}^{-1}g$ is the contraction (${\tilde g}^{ik}g_{kj}$):
The metrics $g$ and $\tilde g$ are c-projectively equivalent iff

 \begin{equation*}
(\nabla_XA)Y=g(X,Y)v_A+Y(\t_A)X+\oo(X,Y)Jv_A-JY(\t_A)JX, 
 \end{equation*}
where $\oo(X,Y)=g(JX,Y)$, $\t_A=\frac14\op{tr}(A)$, $v_A=\op{grad}_g\t_A$.
 \com{
We can also write it so
 \begin{equation}\label{Me2}
(\nabla_Z\hat A)(X,Y)=g(X,Z)\l_A(Y)+g(Y,Z)\l_A(X)+\oo(X,Z)\l_A(JY)+\oo(Y,Z)\l_A(JX), 
 \end{equation}
where $\hat A(X,Y)=g(AX,Y)$, and $\l_A=\frac14d\op{tr}(A)$ ..
 }
In argument-free form\footnote{The same equation serves as a definition of
a Hamiltonian 2-form $\omega(X,AY)$ corresponding to the endomorphism $A$, which attracted a recent interest \cite{ACG}. }
 this writes (using symmetrizer by the last two arguments) so:
 \begin{equation}\label{eqA}
\nabla\hat A=2\op{Sym}_{2,3}\bigl(g\ot\l_A-\oo\ot J^*\l_A\bigr),
 \end{equation}
where $\hat A(X,Y)=g(AX,Y)$ and $\l_A=d\t_A$.

This linear overdetermined PDE system on the unknown $A$ has a finite-dimensional solution space
denoted $\op{Sol}(g,J)$, and $\op{Id}\in\op{Sol}(g,J)$.
{\it Degree of mobility\/} of the pair $(g,J)$ is defined as $D(g,J)=\dim\op{Sol}(g,J)$.

Let us denote by  $\mathfrak{i}(g,J)$ the algebra of $J$-holomorphic infinitesimal isometries of $g$,
by $\mathfrak{h}(g,J)$ the algebra of $J$-holomorphic vector fields that are homotheties for $g$.
We will need the following estimate:
 \begin{lem}\label{L3}
For any pseudo-K\"ahler structure $(g,J)$ we have the inequality: $\dim\mathfrak{cp}(\nabla^g,J) \le\dim\mathfrak{h}(g,J)+D(g,J)-1$.
 \end{lem}

This was discussed in \cite{MR$_2$}, but not formally stated. Though that paper was devoted only to the
K\"ahler metrics the statement is true in general and the proof persists. Let us give a brief argument.
The formula
 $$
A=\phi(v)=g^{-1}\mathcal{L}_v(g)-\tfrac1{2(n+1)}\op{tr}(g^{-1}\mathcal{L}_v(g))\op{Id}
 $$
by \cite{MR$_1$} defines the map
 $$
\phi:\mathfrak{cp}(\nabla^g,J)\to\op{Sol}(g,J)
 $$
and $\op{Ker}(\phi)=\mathfrak{i}(g,J)$. Moreover, if $\pi:\op{Sol}(g,J)\to\op{Sol}(g,J)/\R\cdot\op{Id}$ is the natural projection, then $\op{Ker}(\pi\circ\phi)=\mathfrak{h}(g,J)$. The claim follows from the rank theorem. \qed

 \begin{cor}
We have: $\dim\mathfrak{cp}(\nabla^g,J) \le \dim\mathfrak{i}(g,J)+D(g,J)$. \qed
 \end{cor}

By \cite{DM,MS} the degree of mobility is bounded so:
 $$
D(g,J)\le (n+1)^2,
 $$
and the equality corresponds to spaces of constant holomorphic sectional curvature.
The next biggest dimension, under the additional assumption that there is a
projective non-affine symmetry \cite{MR$_2$,Mi}, is equal to
 \begin{equation}\label{subMaxD}
D_{\text{sub.max}}=(n-1)^2+1=n^2-2n+2.
 \end{equation}

Another ingredient in our count is the estimate on the dimension of the isometry algebra of a K\"ahler structure. Clearly the maximal dimension is $\op{max}\bigl[\dim\mathfrak{i}(g,J)\bigr]=n^2+2n$.

 \begin{prop}\label{P4}
For a K\"ahler structure $g$ of non-constant holomorphic curvature we have: $\dim\mathfrak{i}(g,J)\le n^2+2$.
The bound is sharp and attained, for example, for $M=\C\mathbb{P}^{n-1}\times\C\mathbb{P}^1$.
 \end{prop}

 \begin{proof}
The isotropy algebra of the symmetry algebra is a proper subalgebra of $\mathfrak{u}(n)$, and so is reductive.
All maximal proper subalgebras are $\mathfrak{u}(k)\oplus\mathfrak{u}(n-k)$, and it is clear that the maximal dimension
is attained for $k=1$ or $k=n-1$. Since the Killing vector field is 1-jet determined, we conclude that the sub-maximal
isometry dimension is $2n+(n-1)^2+1=\dim SU(n)+\dim SU(2)=n^2+2$.
 \end{proof}

Now the required bound for $\dim\mathfrak{cp}(\nabla^g,J)$ in the K\"ahler case follows from the Proposition \ref{P4}, Lemma \ref{L3},
the well-known fact that $\mathfrak{h}(g,J)=\mathfrak{i}(g,J)$ if the isometry algebra acts with an open orbit\footnote{Indeed,
if $\vp^*g=\lambda\cdot g$ for a homothety $\vp$ and $x$ is a fixed point with non-zero Riemann curvature tensor $R$, then equality
$\vp^*\|R\|^2=\lambda^{-2}\|R\|^2$ at $x$ implies $\lambda=1$.} 
and formula~(\ref{subMaxD}). However since the latter estimate has an additional assumption \cite{MR$_2$}, we will give  in the next section
another proof in the case there exists no essential projective symmetry (that is a non-affine symmetry for any choice of $g$
in the c-projective class).

\section{Submaximal metric c-projective structures}\label{S7}

By Corollary \ref{Cor} the algebra of c-projective symmetries of a K\"ahler metric is bounded in dimension by $2(n^2-2n+2)$.
The next example shows that this bound is realizable by a metric c-projective structure. Indeed,
consider the following pseudo-K\"ahler metric on $\C^n$ ($J=i$):
 \begin{equation}\label{subMKh}
g=|z_1|^2\,dz_1\,d\bar{z_1}+dz_1\,d\bar{z_2}+d\bar{z_1}\,dz_2
+\sum_{k=3}^n\epsilon_k\,dz_k\,d\bar{z_k}
 \end{equation}
($\epsilon_k=\pm1$).
One easily checks that its Levi-Civita connection coincides with the connection $\nabla$ of type II
given by formula (\ref{subCmax}), and so the sub-maximal complex (integrable $J$) projective
structure is metrizable; in addition by varying the signs $\epsilon_k$ we can achieve any
indefinite signature $(2p,2n-2p)$ for the pseudo-K\"ahler metric, $0<p<n$.

To finish the proof of Theorem \ref{Thm3} we have to show that no K\"ahler metric can have more than $2n^2-2n+3$
linearly independent c-projective symmetries unless it is c-projectively flat (i.e.\ has constant holomorphic sectional curvature).
So we let $g$ be K\"ahler till the end of the proof.
In the case there exists an essential c-projective symmetry (for at least one $g$ with $\nabla^g\in[\nabla]$) the claim follows
from the estimates of Section~\ref{S6}.

Thus let us assume that for a K\"ahler metric $g$ the algebra of c-projective
symmetries coincides with the algebra of (infinitesimal) symmetries of the pair $(\nabla^g,J)$:
$\mathfrak{cp}(\nabla^g,J)=\mathfrak{aff}(g,J)$.

Fix a point $x\in M$ at which the curvature tensor $R$ does not vanish, and consider the
holonomy algebra $\mathcal{H}_x$ of $\nabla^g$ at $x$.
Since $\nabla^g$ preserves both $g$ and $J$, we have $\mathcal{H}_x\subset\mathfrak{u}(n)$.
By Ambrose-Singer theorem $\mathcal{H}_x$ contains the endomorphisms $R(v\wedge w)$ for $v,w\in T_xM$,
so $\mathcal{H}_x\neq0$.

By the (infinitesimal version of) de Rham decomposition theorem \cite{Ei},
we can split $T_xM=\oplus_{k=0}^m\Pi_k$, where $\Pi_0$ is the subspace of complex
dimension $r_0<n$ where $R$ vanishes, and the other pieces are irreducible with respect to $\mathcal{H}_x$
(all $\Pi_k$ are $J$-invariant and so have even real dimensions $2r_k$).
Any K\"ahler metric, which is equivalent to $g$ via a complex affine transformation, is obtained from it
by a block-diagonal automorphism $\op{diag}(A_0,c_1,\dots,c_m)$,
where $A_0\in\op{GL}(\Pi_0,J)$ and $c_k\neq0$ are constant multiples of the identity in the corresponding block.

Therefore the isotropy $\tilde{\mathfrak{a}}_0(x)$ of the complex affine symmetry algebra at $x$ consists of block-diagonal endomorphisms
 $$
\op{diag}(\vp_0,\vp_1,\dots,\vp_m)\in\tilde{\mathfrak{a}}_0(x)\subset\mathfrak{aff}(g,J),
 $$
where $\vp_0\in\mathfrak{gl}(\Pi_0,J)$ is a $\C$-linear matrix of complex size $r_0$
and $\vp_k\in\mathfrak{u}(\Pi_k,g,J)+\R\cdot\op{Id}$ is generated by a unitary transformation of the $k$-th block
of complex size $r_k$ and the standard homothety (scaling of the metric $g$). Consequently we obtain
 $$
\dim\tilde{\mathfrak{a}}_0(x)\leq 2r_0^2+\sum_{k=1}^m(r_k^2+1)\leq 2(n-1)^2+2,
 $$
with equality iff the de Rham decomposition is $(n-1)\times1$ complex block and the smaller block is $\mathfrak{gl}(1,\C)=\mathfrak{u}(1)+\R$
(with a homothety). Next, the upper bound $2n+\tilde{\mathfrak{a}}_0(x)$ on the symmetry algebra $\mathfrak{aff}(g,J)$ is sharp only if the
symmetry acts transitively. As recalled at the end of the last section, a Riemannian metric with an open orbit of the isometry group has no 
essential homotheties. Therefore we get the required estimate
 $$
\dim\mathfrak{aff}(g,J)\leq 2n+ 2(n-1)^2+1=2n^2-2n+3.
 $$
Due to the above arguments (and the fact that the homothety algebra of any non-flat connected 2-dimensional surface has $\dim\leq3$)
it is now clear that this upper bound is achieved iff $(M,g)$ is $\C^{n-1}\times\Sigma^2$, where
$\Sigma^2$ is the complex curve 
equipped with a $J$-compatible metric of constant curvature $K\neq0$.
Theorem \ref{Thm3} is proved.

 \begin{rk}
Let us explain why the submaximal metric structure (\ref{subMKh}) is unique up to an isomorphism.
We use transitivity of the symmetry algebra of the corresponding c-projective structure from Remark \ref{trans}.

Fix a point $o\in M$ ($J$ is also fixed). Then the metric $g$ and the curvature tensor $R_g$
are determined up to complex affine transformation on $T_oM$. In fact, there is an invariant
null-complex line (and correspondingly the dual $\C$-line in the cotangent space) fixed by
the isotropy, and if we fix them the isotropy $\mathfrak{a}_0$ determines $(g,R_g)$ at $o$ uniquely.

Now $(M,g)$ is a symmetric space and so is uniquely determined by the data $(g,R_g)$ at one point $o$.
 \end{rk}

Let us now describe the structure of the symmetric space $M^4_0$ corresponding to the
submaximal (with respect to c-projective transformations) metric $g$ of (\ref{subMKh}) (since the cases $n>2$ are
obtained from this $M^4_0$ by direct product with $\C^{n-2}$, it suffices to study $n=2$ only).
We have $M_0=G/H$ for some Lie groups $G\supset H$ because the symmetry acts transitively.
There are 3 different presentation of $M_0$ as such quotient.

At first, we consider the Lie algebra of c-projective transformations $\mathfrak{s}$ with the isotropy
$\mathfrak{a}_0$ of type II from Section \ref{S2}. This 8D algebra is solvable with the derived series
of dimensions $(8,5,3,0)$. In addition, it has the $\Z_2$-grading:
$\mathfrak{s}=\mathfrak{s}_{-}+\mathfrak{s}_{+}$,
$[\mathfrak{s}_{\epsilon_1},\mathfrak{s}_{\epsilon_2}]\subset\mathfrak{s}_{\epsilon_1\epsilon_2}$.
In a basis $\{e_i\}_{i=1}^4$ of $\mathfrak{s}_{+}=\mathfrak{a}_0$ and a basis $\{e_i\}_{i=5}^8$ of
$\mathfrak{s}_{-}$ the structure equations write:
 \begin{gather*}
[e_1,e_3]=e_3,\ [e_1,e_4]=e_4,\ [e_1,e_5]=2e_5,\ [e_1,e_6]=3e_6,\ [e_1,e_7]=-e_7,\\
[e_2,e_5]=e_5,\ [e_2,e_6]=e_6,\ [e_2,e_7]=-e_7,\ [e_2,e_8]=-e_8,\\
[e_3,e_5]=e_6,\ [e_3,e_7]=e_8,\ [e_4,e_5]=e_6,\ [e_4,e_7]=-e_8,\ [e_5,e_7]=e_3.
 \end{gather*}
Thus $M_0=G^8_c/H^4_c$, where $G^8_c=\op{exp}(\mathfrak{s})$,
$H^4_c=\op{exp}(\mathfrak{s}_{+})$.

Second, consider the group $G^6_k$ of holomorphic isometries, i.e.\ symmetries of the pseudo-K\"ahler
structure $(g,J)$. This group is also solvable, with the derived series of dimensions $(6,5,3,0)$
and the Abelian stabilizer $H^2_k$. Again there is $\Z_2$-grading and the structure equations
of the Lie algebra $\mathfrak{s}'$ in an adapted basis $\{e_i\}_{i=1}^2$ of $\mathfrak{s}'_{+}$,
$\{e_i\}_{i=3}^6$ of $\mathfrak{s}'_{-}$ are:
 \begin{gather*}
[e_1,e_3]=-e_3,\ [e_1,e_4]=e_4,\ [e_1,e_5]=-e_5,\ [e_1,e_6]=e_6,\\
[e_2,e_5]=e_3,\ [e_2,e_6]=e_4,\ [e_5,e_6]=e_2.
 \end{gather*}
Thus $M_0=G^6_k/H^2_k$, where $G^6_k=\op{exp}(\mathfrak{s}')$,
$H^2_k=\op{exp}(\mathfrak{s}'_{+})$.

Finally consider the isometry group $G^8_i$ of $g$ with the stabilizer $H^4_i\simeq\op{GL}(2,\R)$
(the center $e_4$ acts by homothety). The group is the semi-direct product
$SL(2,\R)\ltimes\op{exp}(\mathfrak{r}_5)$, where $\mathfrak{r}_5$ is the 5D Lie algebra given
with the basis $\{e_i\}_{i=4}^8$ given by the relations
 $$
[e_4,e_5]=e_7,\ [e_4,e_6]=e_8,\ [e_5,e_6]=e_4.
 $$
In the basis $\{e_i\}_{i=1}^3$ of $\mathfrak{sl}(2)$: $[e_1,e_2]=-2e_2$, $[e_1,e_3]=2e_3$,
$[e_2,e_3]=e_1$, the representation is given by
 $$
[e_2,e_5]=e_6,\ [e_2,e_7]=e_8,\ [e_3,e_6]=-e_5,\ [e_3,e_8]=-e_7
 $$
and the action of $e_1$ is induced. The $\Z_2$-grading of the resulting
Lie algebra $\mathfrak{s}''=\op{Lie}(G^8_i)$ is $\mathfrak{s}''_{+}=\langle e_i\rangle_{i=1}^4$,
$\mathfrak{s}''_{-}=\langle e_i\rangle_{i=5}^8$, and $M_0=G^8_i/H^4_i$.


\appendix

\section{Some details on the model for type III, $n=2$}\label{S.A}

We found the c-projective structure $(M^4,J,[\nabla])$ of curvature type III with 8 symmetries
using parabolic geometry machinery and Cartan's equivalence method. Our computations exploited
the packages {\tt DifferentialGeometry} and {\tt Cartan} in Maple.

 \subsection{Structure equations}

We first derived the structure equations for the {\em normal} Cartan geometry $(\mathcal{G} \to M,\omega)$ of type $(G, P)$, $n=2$.
For the basics of  parabolic geometry machinery we refer to \cite{CS}.

The curvature 2-form $K = d\omega + \frac{1}{2}[\omega,\omega] \in \Omega^2(\mathcal{G};\g)$ yields the curvature function
$\kappa:\mathcal{G}\to\Lambda^2 \mathfrak{p}_+ \otimes\mathfrak{g}$ via the Killing form identification $(\mathfrak{g}/\mathfrak{p})^*=\fg_1=\mathfrak{p}_+$.  Normality means $\partial^*\kappa=0$, where
$\partial^*: \Lambda^2 \mathfrak{p}_+ \otimes \mathfrak{g} \to \mathfrak{p}_+ \otimes \mathfrak{g}$ is the Kostant codifferential.
Call $\mathbb{K} = \ker\partial^*$ the {\em curvature module}.  The unique grading element $Z \in \g_0 \cong \mathfrak{gl}(2,\mathbb{C})$ stratifies $\mathbb{K}$ into homogeneities, and we decompose each into $\g_0$-irreps:
 \begin{align*}
 \begin{array}{clc}
 \mbox{Homogeneity} & \mbox{$\g_0$-module decomposition} & \dim\\ \hline
 +3 & V_{4(1)} \oplus V_{4(2)} \oplus V_{4(3)} \oplus V_{12} & 24\\
 +2 & V_{16} \oplus V_{8} \oplus V_{6} \oplus V_{2} & 32\\
 +1 & V_{4(4)} & 4
 \end{array}
 \end{align*}
Here, $\dim_\mathbb{R}(V_i) = i$.  The {\em harmonic} curvature corresponds to the modules $V_{4(1)}, V_{16}, V_{4(4)}$, which comprise $\ker(\Box)$, where $\Box = \partial \partial^* + \partial^* \partial$ is the Kostant Laplacian (and $\partial$ is the Lie algebra cohomology differential).

Let $E_{jk}$ denote the $3\times 3$ matrix with a 1 in the $(j,k)$ position and 0 otherwise, and let $F_{jk} = i E_{jk}$.
Decompose into real and imaginary parts, $\omega = \theta + i\eta \in \Omega^1(\mathcal{G};\g)$, and impose a trace-free condition, say $\omega_{22} = -\omega_{11} - \omega_{33}$.  The structure equations are
 \begin{align*}
 d\theta_{jk} &= -\theta_{jl} \wedge \theta_{lk} + \eta_{jl} \wedge \eta_{lk} + \Re(K_{jk})\\
 d\eta_{jk} &= -\eta_{jl} \wedge \theta_{lk} - \theta_{jl} \wedge \eta_{lk} + \Im(K_{jk}).
 \end{align*}
The 2-forms $K_{jk}\in\mathbb{K}$ are obtain by via the duality
 \[
E_{12}\mapsto\theta_{21}, \quad F_{12}\mapsto-\eta_{21}, \quad E_{13}\mapsto\theta_{31}, \quad F_{13}\mapsto-\eta_{31}.
 \]
In particular,
 $$
K_{21} = (A_1 - i A_2) \,\overline{\omega_{21}} \wedge \overline{\omega_{31}} + \dots,\quad
K_{31} = (A_3 - i A_4) \,\overline{\omega_{21}} \wedge \overline{\omega_{31}} + \dots
 $$
where $A_1,...,A_4:\mathcal{G}\to V_{4(4)}$, and similarly for the coefficients $B_1,...,B_{32}$, $C_1,...,C_{24}$ of the other modules $V_{j(k)}$.
The first structure equations are
 \begin{align*}
 dA_1 &= (\theta_{33} - \theta_{11}) A_1 + (5 \eta_{11} + \eta_{33}) A_2 - \theta_{23} A_3 - \eta_{23} A_4 + \alpha_1\\
 dA_2 &= - (5 \eta_{11}+\eta_{33}) A_1 + (\theta_{33} - \theta_{11}) A_2 + \eta_{23} A_3 - \theta_{23} A_4 + \alpha_2\\
 dA_3 &= - \theta_{32} A_1 - \eta_{32} A_2 - (2 \theta_{11}+\theta_{33}) A_3 - (-4 \eta_{11}+\eta_{33}) A_4 + \alpha_3\\
 dA_4 &= \eta_{32} A_1 - \theta_{32} A_2 - (4 \eta_{11}-\eta_{33}) A_3 - (2 \theta_{11}+\theta_{33}) A_4 + \alpha_4
 \end{align*}
 where $\alpha_i$ are semi-basic 1-forms, i.e.\ linear combinations of base forms $\theta_{21}, \eta_{21}, \theta_{31}, \eta_{31}$.
Writing $d\alpha_i = \delta A_i + \alpha_i$, the $\delta A_i$ terms describe the infinitesimal vertical change of these coefficients under the $P$-action.

 \subsection{Derivation of the model}

We follow the method introduced by Cartan \cite{C$_2$}  (for a more detailed explanation see \cite{DMT})
to normalize curvature under the (vertical) action of the structure group.

In our case if the $N_J\neq0$, we obtain the normalization
 \[
 A_1 = 1, \quad A_2 = A_3 = A_4 = 0,
 \]
forcing the relations
 \[
 \theta_{33} = \theta_{11} - \alpha_1, \quad \eta_{33} = -5 \eta_{11} + \alpha_2, \quad
 \theta_{32} = \alpha_3, \quad \eta_{32} = -\alpha_4.
 \]
The residual structure group is now 8-dimensional and still contains $P_+$.  On coefficients in the $V_6$ and $V_2$ modules $P_+$ induces
translation actions on four coefficients, these can all be normalized to zero.  This reduces the bundle to $(\mathcal{E} \to M, S)$, where $S \subset P$ is a 4-dimensional subgroup, and $\mathcal{E}$ comes equipped with:
 \begin{itemize}
 \item an $S$-equivariant coframing: $\omega_{21}, \omega_{31}, \omega_{11}, \omega_{23}$;
 \item a vertical distribution $\mathcal{V}=\langle\omega_{21}, \omega_{31}\rangle^\perp$;
 \item an $S$-connection $\gamma$ with horizontal distribution $\mathcal{H}=\langle\omega_{11}, \omega_{23}\rangle^\perp$.
 \end{itemize}
The symmetry algebra of the c-projective structure is bounded by $8=\dim\mathcal{E}$. For this bound to be sharp, $S$ must act trivially on curvature coefficients. This forces the vanishing of many parts of the curvature function.  Indeed, after resolving all integrability conditions, we found that there is a {\em unique} model with 8 symmetries, and for it $\kappa|_{\mathcal{E}}$ has non-trivial components only in $V_{4(4)}$ and $V_8$.
Here are the structure equations:
 \begin{align*}
 d\omega_{21} &= -\omega_{23} \wedge \omega_{31} + 3 \overline{\omega_{11}} \wedge \omega_{21} + 6 \overline{\omega_{21}} \wedge \overline{\omega_{31}}\\
 d\omega_{31} &= 6i \mathfrak{Im}(\omega_{11}) \wedge \omega_{31} \\
 d\omega_{11} &= 3 \omega_{31} \wedge \overline{\omega_{31}} \\
 d\omega_{23} &= -27 \omega_{21} \wedge \overline{\omega_{31}} + 3\omega_{11} \wedge \omega_{23} - 12 i\mathfrak{Im}(\omega_{11}) \wedge \omega_{23}
 \end{align*}
The embedding relations and the structure algebra are:
 \begin{align*}
 & \omega_{12} = \omega_{32} = 0, \quad
 \omega_{13} = 15 \overline{\omega_{31}}, \quad
 \omega_{33} = \omega_{11} - 6i\mathfrak{Im}(\omega_{11})
 \end{align*}
  \begin{equation}\label{SAM}
 \begin{pmatrix}
 a_0 + i a_1 & 0 & 0 \\
 0 & -2a_0 + 4ia_1 & b_0 + i b_1\\
 0 & 0 & a_0 - 5ia_1
 \end{pmatrix}, \quad a_i, b_i \in \mathbb{R}.
 \end{equation}

Let $W_{jk} = e_{jk} + i f_{jk}$ be the dual framing on $\mathcal{E}$.  Then
 \begin{align*} \label{E:J}
 f_{21} \otimes \theta_{21} - e_{21} \otimes \eta_{21} + f_{31} \otimes \theta_{31} - e_{31} \otimes \eta_{31}
 \end{align*}
is pullback of the almost complex structure $J$ on $M$.  The minimal complex connection $\nabla$ can be read off from the principal connection $\gamma$.

Viewing $TM \cong \mathcal{E} \times_S V\simeq\mathbb{R}^4$ we integrate the structure equations and
obtain the model in coordinates as indicated in Section \ref{S3}.

 \subsection{Deformation approach}

Another approach to get a sub-maximal model is to deform a graded sub-algebra of $\g$
by preserving its filtered Lie algebra structure, but destroying the grading \cite{K$_3$,KT}.

In our case the graded sub-algebra $\mathfrak{a}^\phi=\g_{-}\oplus\mathfrak{a}_0\subset\g$
has complex matrix representation ($\a=\a^1+i\a^2, \b=\b^1+i\b^2, \nu_k=\nu_k^1+i\nu_k^2\in\C$)
 $$
\mathfrak{a}^\phi\ni A=\left(\begin{array}{c|cc}
 \a & 0 & 0 \\ \hline
 \nu_1 & 3\bar{\a}-2\a & 0\\
 \nu_2 & \b & \a-3\bar{\a}
 \end{array} \right)
 $$
(notice lower-triangular form vs. the upper-triangular form for the $2\times2$ block of the structure algebra in (\ref{SAM})
operating with the highest weight vector; they are conjugate by interchanging indices 2,3)
and we get basis by decomposition $A=\sum_{j=1}^2(\a^ja_j+\b^jb_j+\nu_j^1v'_j+\nu_j^2v''_j)$;
$\g_{-}=\langle v'_1,v''_1,v'_2,v''_2\rangle$ has grade $-1$ and
$\mathfrak{a}_0=\langle a_1,a_2,b_1,b_2\rangle$ has grade $0$.

The algorithm of deforming the Lie algebra structure on $\mathfrak{a}^\phi$ via the lowest weight vector
here fails. However the deformation exists. To find it let us deform the structure constants respecting the filtration
on $\mathfrak{a}^\phi$ (i.e.\ brackets $[\mathfrak{a}_0,\mathfrak{a}_0]$ are fixed, $[\g_{-},\mathfrak{a}_0]$ can
be changed by $\mathfrak{a}_0$ and $[\g_{-},\g_{-}]$ can be changed by everything), the Jacobi identity imposed.

This deformation has several branches (some with other types curvature), one of which is
($\lambda$ is the deformation parameter):
 \begin{gather*}
[a_1,b_1] = -3b_1, \ [a_1,b_2] = -3b_2, \ [a_2,b_1] = 9b_2, \ [a_2,b_2] = -9b_1,\\
[a_1,v'_2] = -3v'_2, \ [a_1,v''_2] = -3v''_2, \ [a_2,v'_1] = -6v''_1, \ [a_2,v''_1] = 6v'_1, \\
[a_2,v'_2] = 3v''_2, \ [a_2,v''_2] = -3v'_2, \ [b_1,v'_1] = v'_2, \ [b_1,v''_1] = v''_2, \\
[b_2,v'_1] = v''_2, \ [b_2,v''_1] = -v'_2, \ [v'_1, v''_1] = 6 \lambda^2 a_2, \\
[v'_1,v'_2] = 6\lambda v'_2-27\lambda^2 b_1, \ [v'_1,v''_2] = -6\lambda v''_2-27\lambda^2 b_2, \\
[v''_1,v'_2] = -6\lambda v''_2 +27\lambda^2 b_2, \ [v''_1,v''_2] = -6\lambda v'_2 -27\lambda^2 b_1.
 \end{gather*}
(notice that in the non-graded case $\lambda\neq0$, we can rescale $\lambda=1$).

These relations determine the Lie algebra $\mathfrak{f}$ with subalgebra
$\mathfrak{k}=\langle a_1,a_2,b_1,b_2\rangle\simeq\mathfrak{a}_0$.
The simply-connected Lie group $F$ of $\mathfrak{f}$ contains the Lie subgroup $K$
with $\op{Lie}(K)=\mathfrak{k}$.
As such we can take the normalizer in $F$ of $\g_{-}$ with respect to the adjoint action (for $\lambda\ne0$).

The homogeneous space $M^4=F/K$ has $F$-invariant (non-integrable) almost complex structure $J$
given by $Jv_j'=v_j''$. Moreover $M$ has $F$-invariant projective connection $[\nabla]$ given by
the Cartan bundle construction \cite[Lemma 4.1.4]{KT}. Since for non-zero values of the
parameters the symmetry algebra is non-graded, the obtained c-projective structure $(J,[\nabla])$ is
not-flat ($\mathfrak{a}^\phi$ is not filtration-rigid, see \cite[Proposition 4.2.2]{KT})
and hence it is sub-maximal symmetric with the symmetry algebra $\op{Sym}([\nabla],J)=\mathfrak{f}$ of dimension 8.

\section{Uniqueness of the submaximal structures}\label{S.B}

Classification of submaximal symmetric structures can be an extrem\-ely difficult problem depending
on the geometry in question\footnote{Maximal symmetric structures are unique in parabolic geometry,
but description of all such for more general structures can be quite intricate.}.
There does not exist any general result in this direction in the literature, but for c-projective structures
we can confirm the uniqueness as follows.

 \subsection{Classification of submaximal c-projective structures}

Submaximal c-projective structure of type II is unique up to an isomorphism.
Indeed, by the result of Section \ref{S2} the stabilizer of the symmetry algebra
(up to isomorphism) is equal to $\mathfrak{a}_0$. As explained in \cite{K$_3$,KT},
the symmetry algebra $\mathfrak{s}$ is filtered with the corresponding graded algebra
being $\mathfrak{a}^\phi=\g_-\oplus\mathfrak{a}_0$.

The process of recovery of $\mathfrak{s}$ from its subalgebra $\mathfrak{a}_0$
and the action of $\mathfrak{a}_0$ on $\mathfrak{s}/\mathfrak{a}_0=\g_-$ is described
as follows: 
one has to introduce indeterminate coefficients of the undetermined
commutators and then constrain these coefficients by the Jacobi identity.

Though in general this is quite a complicated system of quadratic equations, in our case
many linear equations that occur allow to resolve it. We used Maple to facilitate the heavy
computation. All cases $n\ge2$ follow the same pattern,
and as the output we obtain a 1-parameter Lie algebra structure $\tilde{\mathfrak{s}}(t)$.

If the parameter $t=0$ we get the graded algebra $\mathfrak{a}^\phi$, while the case
$t\neq0$ reparametrizes to $t=1$ corresponding to the symmetry $\mathfrak{s}$ of (\ref{subCmax}).
Thus there are only two cases to consider.

Contrary to the non-exceptional parabolic geometries the possibility of graded symmetry algebra
does not imply flatness of the c-projective structure (compare \cite[Example 4.4.3]{KT}).

Therefore on the next step we look for c-projective structures invariant with respect to
$\mathfrak{a}^\phi=\tilde{\mathfrak{s}}(0)$ and $\mathfrak{s}=\tilde{\mathfrak{s}}(1)$.
Such structures are unique and are: flat and (\ref{subCmax}) in the first/second cases respectively.
This proves the claim for type II structures.

\smallskip

Submaximal c-projective structure of type I is also unique up to an isomorphism.
This is actually a holomorphic version of the uniqueness of Egorov's submaximal (real) projective
structure. Such result was expected by experts, but Egorov's paper \cite{E$_1$} does not contain
an indication of this result. Therefore we have verified it directly by the method described for type II
(note that our computation applies to both smooth real and complex analytic cases).

Here everything is similar, but the pattern holds for the cases $n>2$ and the Maple computation
asserts the result. The case $n=2$ is an exception, and we refer the reader to \cite{Tr,K$_2$} for
the discussion of the smooth situation, in which case there are two submaximal models.
The analytic case is similar but the two models glue because $\pm$ arising in the smooth case
can be renormalized over $\C$. As the conclusion we obtain unicity for type I submaximal structures.

\medskip

A computer verification shows that submaximal c-projective structure of type III is also
unique up to an isomorphism for $n>2$, but the case $n=2$ for type III is an exception, and here
the uniqueness of the submaximal c-projective structure follows from the Cartan equivalence
method as described in Appendix \ref{S.A}.

 \subsection{Metrics with the submaximal c-projective symmetry}

Let us compute all metrics c-projectively equivalent to the pseudo-K\"ahler metric $g$ given by (\ref{subMKh}).
These are precisely those metrics that solve the metrizability equation for the c-projective structure (\ref{subCmax}), and
their number is equal to the dimension of the solution space of (\ref{eqA}), i.e.\
the degree of mobility of this metric: $D(g,J)=D_{\op{sub.max}}=(n-1)^2+1$.

To find these equivalent metrics let us compute the space of all parallel 1-forms:
 \begin{equation}\label{qwe}
dz_2, \dots, dz_n;\ d\bar{z}_2, \dots, d\bar{z}_n
 \end{equation}
(these already split into $(1,0)$ and $(0,1)$ type respective to $J$). Since
$\mathfrak{cp}(\nabla^g,J)=\mathfrak{aff}(g,J)$ the required metrics are linear
combinations of $g$ and $(1,1)$-type quadrics in the forms (\ref{qwe})
(the coefficient of $g$ in such combination has to be nonzero by nondegeneracy).

Indeed, dimension of the space of such combinations is $(n-1)^2+1$, and since this number equals $D_{\op{sub.max}}$,
there exists no other metric that is complex affine equivalent to the metric $g$. 
Thus the general metric, c-projectively equivalent to $g$ (up to scaling) is equal to
 $$
\hat{g}=|z_1|^2\,dz_1\,d\bar{z_1}+dz_1\,d\bar{z_2}+d\bar{z_1}\,dz_2
+\sum_{k,l=2}^nc_{kl}\,dz_k\,d\bar{z_l} \quad (c_{lk}=\overline{c_{kl}}),
 $$
and we again confirm that a pseudo-K\"ahler metric $\hat{g}$ with submaximal number of
c-projective symmetries cannot be of Riemannian signature.

\bigskip

{\bf\sc Acknowledgements.}
We thank A.\,\v{C}ap for helpful discussion and the anonymous referee for pointing some inconsistencies in the 
initial version of the text. The research of B.K.\ and V.M.\ was supported by German DAADppp grant 50966389 and the Norwegian 
Research Council. The research of V.M.\ and D.T.\ was supported by Go8-DAAD grant 56203040 of Germany-Australia cooperation.
D.T.\ is also supported by project M1884-N35 of the Austrian Science Fund (FWF).


\end{document}